\documentclass[12pt,a4paper]{amsart}
\usepackage{amsfonts,color}
\usepackage{amsthm}
\usepackage{amsmath}
\usepackage{amscd}
\usepackage[utf8]{inputenc}
\usepackage{t1enc}
\usepackage[mathscr]{eucal}
\usepackage{indentfirst}
\usepackage{graphicx}
\usepackage{graphics}
\usepackage{pict2e}
\usepackage{epic}
\usepackage{url}
\usepackage{epstopdf}
\usepackage{comment}
\usepackage{todonotes}
\usepackage{hyperref}
\usepackage{subcaption}
\usepackage{amssymb}

\numberwithin{equation}{section}
\usepackage[margin=2.6cm]{geometry}

\usepackage{pgfplots}
\usepackage{xcolor}
\usepackage{tikz}
\usetikzlibrary{matrix,arrows,decorations.pathmorphing}
\usetikzlibrary{calc,decorations.pathreplacing}
\usetikzlibrary{quotes,angles}
\usetikzlibrary{shapes}
\usetikzlibrary{patterns}

\tikzstyle{vertex}=[draw=black,circle,fill=black,minimum size=6pt, inner sep=0pt, outer sep=0pt,text=black,line width=0mm]

\tikzstyle{Sqvertex}=[draw=black,shape=rectangle, minimum size=10pt, fill=white]

\tikzstyle{Cvertex}=[draw=black,shape=circle, minimum size=6pt, fill=white]

\tikzstyle{vertex_blue}=[draw=black,circle,fill=blue,minimum size=6pt, inner sep=0pt, outer sep=0pt,text=black,line width=0mm]
\tikzstyle{vertex_red}=[draw=black,circle,fill=red,minimum size=6pt, inner sep=0pt, outer sep=0pt,text=black,line width=0mm]
\tikzstyle{vertex_green}=[draw=black,circle,fill=green,minimum size=6pt, inner sep=0pt, outer sep=0pt,text=black,line width=0mm]
\tikzstyle{c0}=[shape=circle, minimum size=4pt, fill=white]
\tikzstyle{c1}=[shape=rectangle, minimum size=7pt, fill=red]
\tikzstyle{c2}=[shape=diamond, minimum size=10pt, fill=blue]
\tikzstyle{mybox} = [rectangle, rounded corners, minimum width=3cm, minimum height=1cm,text centered, draw=black]
\tikzset{base/.style = {rectangle, rounded corners, draw=black,
                           minimum width=3cm, minimum height=1cm,
                           text centered}}

\usepgfplotslibrary{fillbetween}
\pgfplotsset{mystyle/.style={%
        xmin=-2,
        xmax=7.9,
        ymin=-1,
        xtick = {1,3},
        xticklabels = {{1},$d-1$},
        ytick = {1}
    }
}

\definecolor{darkerblue}{HTML}{065A82} 
\definecolor{lighterblue}{HTML}{1C7293} 

\theoremstyle{plain}
\newtheorem{Th}{Theorem}[section]
\newtheorem{Lemma}[Th]{Lemma}
\newtheorem{Cor}[Th]{Corollary}

 \theoremstyle{definition}
\newtheorem{Def}[Th]{Definition}
\newtheorem{Conj}[Th]{Conjecture}
\newtheorem{Rem}[Th]{Remark}
\newtheorem{?}[Th]{Problem}

\newcommand{\E}{\mathbb{E}}
\newcommand{\ia}{\mathrm{ia}}
\newcommand{\ea}{\mathrm{ea}}

\newcommand{\wT}{\widetilde{T}}

\begin{document}

\title{Around the Merino--Welsh conjecture: improving Jackson's inequality}

\author[P. Csikv\'ari]{P\'{e}ter Csikv\'{a}ri}

\address{HUN-REN Alfr\'ed R\'enyi Institute of Mathematics, H-1053 Budapest Re\'altanoda utca 13-15 \and ELTE: E\"{o}tv\"{o}s Lor\'{a}nd University  Mathematics Institute, Department of Computer
Science  H-1117 Budapest, 
P\'{a}zm\'{a}ny P\'{e}ter s\'{e}t\'{a}ny 1/C}

\email{peter.csikvari@gmail.com}

\thanks{The research was supported by the MTA-R\'enyi Counting in Sparse Graphs ''Momentum'' Research
Group,  by the Hungarian National Research, Development and Innovation
Office, Advanced grant 153378, and by Dynasnet European Research Council Synergy project -- grant number ERC-2018-SYG 810115.}

\begin{abstract}
The Merino-Welsh conjecture states that for a graph $G$ without loops and bridges the Tutte polynomial $T_G(x,y)$ satisfies the inequality
$$\max(T_G(2,0),T_G(0,2))\geqslant T_G(1,1).$$
Later Jackson proved that for any matroid $M$ without loops and coloops we have
$$T_M(3,0)T_M(0,3)\geqslant T_M(1,1)^2.$$
The value $3$ in this statement was improved to $2.9243$ by Beke, Cs\'aji, Csikv\'ari and Pituk. In this paper, we further improve on this result by showing that 
$$T_M(2.355,0)T_M(0,2.355)\geqslant T_M(1,1)^2.$$
We also prove that the Merino--Welsh conjecture is true for matroids $M$, where all circuits of $M$ and its dual $M^*$ have length between $\ell$ and $(\ell-2)^2(\ell^2-4\ell+2)$ for some $\ell\geqslant 4$.
\end{abstract}

\maketitle

\section{Introduction}
Let $G$ be a connected graph without loops and bridges. Merino and Welsh \cite{merino1999forests} conjectured that
\begin{equation*}
\max\left(\alpha(G), \alpha^*(G)\right) \geqslant \tau(G),
\end{equation*}
where $\alpha(G), \alpha^*(G),\tau(G)$  denote the number of acyclic orientations, strongly connected orientations, and spanning trees of $G$, respectively. These quantities are evaluations of the Tutte polynomial $T_G(x,y)$, namely $T_G(2,0)=\alpha(G)$, $T_G(0,2)=\alpha^*(G)$ and $T_G(1,1)=\tau(G)$. 

Conde and Merino \cite{conde2009comparing} proposed ``additive'' and ``multiplicative'' versions of this conjecture:
\begin{align*}
T_G(2,0) + T_G(0,2) &\geqslant 2T_G(1,1), \\
T_G(2,0)T_G(0,2) &\geqslant T_G(1,1)^2,  
\end{align*}
respectively. The multiplicative version implies the additive version, which in turn implies the original conjecture. These conjectures also naturally extend to the Tutte polynomial of a matroid without loops and coloops.

While the conjecture holds for certain classes of graphs \cite{lin2013note, noble2014merino, thomassen2010spanning} and matroids \cite{chavez2011some,ferroni2023merino,knauer2018tutte,kung2021inconsequential,merino2009note} it fails for general matroids. Beke, Cs\'aji, Csikv\'ari and Pituk \cite{beke2024merino} showed that there exist infinitely many matroids without loops and coloops violating the multiplicative version.

\begin{Th}[Beke, Cs\'aji, Csikv\'ari and Pituk \cite{beke2024merino}] \label{counter example}
There are infinitely many matroids $M$ without loops and coloops for which 
$$T_M(2,0)T_M(0,2)<T_M(1,1)^2.$$
Furthermore, let $x_0$ be the largest root of the polynomial $x^3-9(x-1)$. ($x_0\approx 2.22668...$)  Then for any $0\leqslant a<x_0$ there are infinitely many matroids $M$ without loops and coloops for which 
$$T_M(a,0)T_M(0,a)<T_M(1,1)^2.$$
\end{Th}

This paper investigates the following question: for what values of a does the inequality
\begin{equation*}
T_M(a,0)T_M(0,a) \geqslant T_M(1,1)^2
\end{equation*}
hold for all loopless and coloopless matroids $M$? In this direction the first major result is due to Jackson.

\begin{Th}[Jackson \cite{jackson2010inequality}]
For any matroid M without loops and coloops,
\begin{equation*}
T_M(3,0)T_M(0,3) \geqslant T_M(1,1)^2.
\end{equation*}
\end{Th}

Jackson's result was improved by Beke, Cs\'aji, Csikv\'ari and Pituk in the paper \cite{beke2024permutation}. They showed that one can write $2.9243$ instead of $3$. In this paper, we further improve on this inequality.

\begin{Th} \label{main theorem}
For any matroid M without loops and coloops and $a\geqslant2.355$,
\begin{equation*}
T_M(a,0)T_M(0,a) \geqslant T_M(1,1)^2.
\end{equation*}
\end{Th}

We may also study which matroid classes satisfy the product version of the Merino--Welsh conjecture. The following theorem is motivated by the fact that paving matroids satisfy the Merino--Welsh conjecture \cite{chavez2011some,ferroni2023merino}. A matroid of rank $r$ is a paving matroid if all circuits have length $r$ or $r+1$. 

\begin{Th} \label{circuit length}
Suppose that there exists an $\ell\geqslant 4$ such that all circuits of the matroid $M$ and its dual $M^*$ have length between $\ell$ and $(\ell-2)^2(\ell^2-4\ell+2)$. Then
$$T_M(2,0)T_M(0,2)\geqslant T_M(1,1)^2.$$
\end{Th}

The proofs of Theorem~\ref{main theorem} and Theorem~\ref{circuit length} are based on the theory of the permutation Tutte polynomials developed in the paper \cite{beke2024permutation}. While Theorem~\ref{circuit length} does not imply that paving matroids satisfy the Merino--Welsh conjecture,  one can prove this fact by modifying the proof of Theorem~\ref{circuit length}. 
\bigskip

\noindent \textbf{Notation.} Throughout the paper $G=(V,E)$ is an arbitrary graph and $H=(A,B,E)$ is a bipartite graph. $K_{a,b}$ denotes the complete bipartite graph with parts of size $a$ and $b$. $S_k$ denotes the star graph on $k$ vertices, that is, $S_k=K_{1,k-1}$. For a vertex $v$ the degree of $v$ is denoted by $d_v$. $N_H(v)$ denotes the set of neighbors of $v$. If $H$ is clear from the context, then we simply write $N(v)$.

\bigskip

\noindent \textbf{This paper is organized as follows.} 
\begin{itemize}
\item In the next section, we introduce the basic concepts from matroid theory that we will use, and revisit the theory of the permutation Tutte polynomial $\wT_H(x,y)$ developed in the paper \cite{beke2024permutation}. 
\item In Section~\ref{sect: warm-up} we compute the growth constant $\lim_{n\to \infty}\wT_{H_n}(x,0)^{1/n}$ for some families of bipartite graphs. While this section is not necessary for the proofs of Theorem~\ref{main theorem} and \ref{circuit length} it provides an important intuition concerning the permutation Tutte polynomial. 
\item In Section~\ref{sect: main_lemma} we prove a technical, but very important lemma that provides the basis of the proof of Theorems~\ref{main theorem} and \ref{circuit length}.
\item In Section~\ref{sect: proof_main_theorem} we give the proof of Theorem~\ref{main theorem}. \item In Section~\ref{sect: circuit_length} we prove Theorem~\ref{circuit length}.
\item In Section~\ref{sect: concluding_remarks} we end the paper with some conjectures.
\item In the Appendix one can find some tables that are used in the proof of Theorem~\ref{main theorem}.
\end{itemize}

\section{Preliminaries}
\label{preliminaries}

This section collects the necessary tools from matroid theory, and recalls some of the basic facts from the theory of the permutation Tutte polynomial.

\subsection{Tutte polynomial and matroids}
The Tutte polynomial $T_G(x,y)$ of a graph $G$ is defined as
$$T_G(x,y)=\sum_{A\subseteq E}(x-1)^{k(A)-k(E)}(y-1)^{k(A)+|A|-v(G)},$$
with $k(A)$ denoting the number of connected components of the graph $(V,A)$ and $v(G)$ denotes the number of vertices of $G$, see \cite{tutte1954contribution}.
There is a vast literature on the properties of the Tutte polynomial and its applications, see for instance, \cite{brylawski1992tutte,crapo1969tutte,ellis2011graph,welsh1999tutte}, or the book \cite{ellis2022handbook}.

The Tutte polynomial naturally extends to matroids. Recall that a matroid $M$ is a pair $(E,\mathcal{I})$ such that $\mathcal{I}\subseteq 2^{E}$, the collection of independent sets,  satisfies the axioms 
\begin{itemize}
\item[(i)] $\emptyset \in \mathcal{I}$, 
\item[(ii)] if $A'\subseteq A\in \mathcal{I}$, then $A'\in \mathcal{I}$, 
\item[(iii)] if $A,B\in \mathcal{I}$ such that $|B|<|A|$, then there exists an $x\in A\setminus B$ such that $B\cup \{x\} \in \mathcal{I}$. 
\end{itemize}
It turns out that given a set $S\subseteq E$, the maximal independent subsets of $S$ all have the same cardinality. This cardinality is called the rank of the set $S$, and is denoted by $r(S)$. The maximum size independent sets of $M$ are called bases, and the set of bases is denoted by $\mathcal{B}(M)$. The dual of a matroid $M$ is the matroid  $M^*$ whose bases are $\{E\setminus B\ |\ B\in \mathcal{B}(M) \}$. A circuit of a matroid is a minimal dependent set $A$, that is, each proper subset of $A$ is an independent set, but $A$ itself is not in $\mathcal{I}$. For further details on matroids, see for instance \cite{oxley1992matroid}.
\medskip

This paper assumes very little familiarity with matroid theory, and we offer the following two  families of examples for those unfamiliar with the basics. 

Let $U_{r,n}$ denote the uniform matroid of rank $r$ on $n$ elements: here the set $\mathcal{I}$ consists of the subsets of $\{1,2,\dots ,n\}$ that have size at most $r$. Clearly, $\mathcal{I}$ satisfies all three axioms. Here the bases are the subsets of $\{1,2,\dots ,n\}$ that have size exactly $r$. The circuits are the sets of size $r+1$, and the dual matroid is $U_{n-r,n}$.

Another important family of matroids are the so-called cycle matroids. Given a graph $G=(V,E)$ we can consider the matroid $M_G$ on the ground set $E$, where $A\subseteq E$ is independent if it does not contain a cycle, that is, it is a forest. It is clear that it satisfies (i) and (ii), and a little inspection shows that it also satisfies (iii). This is called the cycle matroid of $G$, and is denoted by $M_G$. One can see that if $G$ is connected, then the bases of $M_G$ are the spanning trees of $G$. In general, the rank of a set $A\subseteq E$ is $v(G)-k(A)$, where $k(A)$ is the number of connected components of the subgraph $(V,A)$. A circuit of $M_G$ corresponds to the edge set of a cycle of $G$. We believe that  understanding uniform matroids and cycle matroids is sufficient to have a good grasp on the content of this paper.
\bigskip

One can define the Tutte polynomial of a matroid 
 as $$T_M(x,y)=\sum_{S\subseteq E}(x-1)^{r(E)-r(S)}(y-1)^{|S|-r(S)},$$
where $r(S)$ is the rank of a set $S\subseteq E$. When $M=M_G$, then $T_{M_G}(x,y)=T_G(x,y)$. A loop in a matroid $M$ is an element $x\in E$ such that $r(\{x\})=0$, that is, $\{x\}\notin \mathcal{I}$, and a coloop is an element that is a loop in the dual $M^*$ of the matroid $M$. Equivalently, a coloop is an element that is in every base of $M$. For a cycle matroid $M_G$, loops correspond to loop edges, and coloops correspond to bridges in the graph $G$. 
\bigskip

Hence, it was suggested that the inequalities  
$$\max(T_M(2,0),T_M(0,2)\geqslant T_M(1,1),$$
$$T_M(2,0)+T_M(0,2)\geqslant 2T_M(1,1),$$
$$T_M(2,0)T_M(0,2)\geqslant T_M(1,1)^2$$
may hold true for all matroids $M$ without loops and coloops. (These versions appear explicitly in \cite{ferroni2023merino}, but were treated much earlier without explicitly calling them conjectures.) 
 Note that for general matroids, all these variants are equivalent in the following sense: if one of them is true for all matroids, then the others are also true for all matroids. Applying the maximum version to $M\oplus M^*$ with $M^*$ being the dual of $M$ leads to the multiplicative version of the conjecture. (Here $M\oplus N$ denotes the direct sum of the matroids $M$ and $N$.)

\subsection{Permutation Tutte polynomial}
The proof of Theorem~\ref{main theorem} heavily relies on the theory of the permutation Tutte polynomial. The idea is that the Tutte polynomial $T_G(x,y)$ can be written as a sum of the permutation Tutte polynomials $\widetilde{T}_{H_j}(x,y)$ for certain bipartite graphs $H_j$. As a consequence certain inequalities valid for the permutation Tutte polynomial transfer to the Tutte polynomial.

\begin{Def}[\cite{beke2024permutation}] \label{main-def}
Let $H=(A,B,E)$ be a bipartite graph. Suppose that $V(H)=[m]$. For a permutation $\pi:[m]\to [m]$, we say that a vertex $i\in A$ is internally active if
$$\pi(i)>\max_{j\in N_H(i)}\pi(j),$$
where the maximum over an empty set is set to be $-\infty$.
Similarly, we say that vertex $j\in B$ is externally active if
$$\pi(j)>\max_{i\in N_H(j)}\pi(i).$$
Let $\ia(\pi)$ and $\ea(\pi)$ be the number of internally and externally active vertices in $A$ and $B$, respectively.
Let
$$\widetilde{T}_H(x,y)=\frac{1}{m!}\sum_{\pi \in S_m}x^{\ia(\pi)}y^{\ea(\pi)},$$
where $S_m$ denotes the set of all permutations on $m$ elements.
We will call $\widetilde{T}_H(x,y)$ the permutation Tutte polynomial of $H$.
\end{Def}

The above definition is motivated by the following theorem of Tutte.

\begin{Th}[Tutte \cite{tutte1954contribution}] \label{ia-ea-characterization}
Let $G$ be a connected graph with $m$ edges. Label the edges with $1,2,\dots,m$ arbitrarily. In the case of a spanning tree $T$ of $G$, let us call an edge $e\in E(T)$ internally active if $e$ has the largest label among the edges $f'\in E(G)$ that are going between the connected components of the graph $T-e$. Let us call an edge $f\notin E(T)$ externally active if $e$ has the largest label among the edges in the cycle determined by $T$ and $f$ by adding $f$ to $T$. Let $\mathrm{ia}(T)$ and $\mathrm{ea}(T)$ be the number of internally and externally active edges, respectively. Then
$$T_G(x,y)=\sum_{T\in \mathcal{T}(G)}x^{\mathrm{ia}(T)}y^{\mathrm{ea}(T)},$$
where the summation runs over all spanning trees of $G$.
\end{Th}

Theorem~\ref{ia-ea-characterization} was originally a definition for the Tutte polynomial \cite{tutte1954contribution}. This theorem or definition naturally extends to matroids even though Tutte's original paper only concerned graphs. This characterization of the Tutte polynomial immediately shows that the coefficients of the Tutte polynomial are non-negative. In this theorem, we are restricted to the same labeling of the edges for all spanning trees. For those who have never seen this definition before, it might be very surprising that the Tutte polynomial is independent of the actual choice of the labeling.

To explain the connection between $T_G(x,y)$ and $\widetilde{T}_H(x,y)$, we need the concept of the local basis exchange graph.

\begin{Def} The local basis exchange graph $H[T]$ of a graph $G=(V,E)$ with respect to a spanning tree  $T$ is defined as follows.
 The graph $H[T]$ is a bipartite graph whose vertices are the edges of $G$. One bipartite class consists of the edges of $T$, the other consists of the edges of $E\setminus T$, and we connect a spanning tree edge $e$ with a non-edge $f$ if $f$ is in the cut determined by $e$ and $T$, equivalently, $e$ is in the cycle determined by $f$ and $T$. This is also equivalent with $T-e+f$ being a spanning tree of $G$ again.
 
 Clearly, this definition works for general matroids $M=(E,\mathcal{I})$ and their basis. 
 If $A$ is a basis, then let $H[A]$ be the bipartite graph with $V(H[A])=E$,  where one part consists of the elements of $A$, the other part consists of $E\setminus A$, and $e\in A$ and $f\in E\setminus A$ are adjacent in the bipartite graph $H[A]$ if $A-e+f$ is again a basis.
\end{Def}
Figure 1 shows a graph $G$ with a spanning tree $T$ and the bipartite graph $H[T]$ obtained from $T$. 
\bigskip

For a fixed labeling of the edges of $G$, we get a labeling of the vertices of $H[T]$, and the internally (externally) active edges of $G$ correspond to internally (externally) active vertices of $H[T]$, so the two definitions of internal and external activity are compatible. The following lemma is crucial for us, so we even included its proof.

\begin{figure}[htp] 
\begin{tikzpicture}[, scale=0.33, baseline=0pt, node distance={20mm}, thick, main/.style = {draw, circle, fill=black}] 
\node[main] (1) {}; 
\node[main] (2) [above right of=1] {}; 
\node[main] (3) [below right of=2]{}; 
\node[main] (4) [below of=1]{}; 
\node[main] (5) [below of=3]{}; 
\node[main] (6) [below right of=4]{}; 
\draw [color=blue,line width=2pt](1) edge node[
above,black]{$1$} (2) ; 
\draw [color=blue, line width=2pt](1) edge node[pos=0.3, below, black]{$2$} (3) ; 
\draw [color=blue, line width=2pt](1) edge node[left, black]{$3$} (4) ; 
\draw [color=red, line width=2pt](2) edge  node[pos=0.15, right, black]{$4$} (5) ; 
\draw [color=red, line width=2pt](3) edge  node[pos=0.4, left, black]{$6$} (6) ; 
\draw [color=blue, line width=2pt](4) edge node[pos=0.3, above, black]{$5$} (5) ; 
\draw [color=blue, line width=2pt](4) edge  node[below, black]{$7$} (6) ; 
\draw [color=red, line width=2pt](5) edge  node[below, black]{$8$} (6) ; 
\end{tikzpicture} 
\qquad \qquad
\begin{tikzpicture}[, scale=0.33, baseline=0pt, node distance={18mm}, thick, main/.style = {draw, circle}]
\node[main, fill=blue, label=$1$] (1) {}; 
\node[main, fill=blue, label=$2$] (2) [right of=1]{};
\node[main,fill=blue, label=$3$] (3) [right of=2]{}; 
\node[main,fill=blue, label=$5$] (4) [right of=3]{}; 
\node[main,fill=blue, label=$7$] (5) [right of=4]{}; 
\node[main,fill=red, label={[yshift=-30pt]$4$}] (6) [below of=2]{}; 
\node[main,fill=red, label={[yshift=-30pt]$6$}] (7) [below of=3]{};
\node[main,fill=red, label={[yshift=-30pt]$8$}] (8) [below of=4]{};
\draw (1) -- (6) ; 
\draw (2) -- (7) ; 
\draw (3) -- (6) ; 
\draw (3) -- (7) ; 
\draw (4) -- (6) ; 
\draw (4) -- (8) ; 
\draw (5) -- (7) ; 
\draw (5) -- (8) ; 
\end{tikzpicture} 
\caption{Example for a graph $G$ and the local basis exchange graph $H[T]$ obtained from a spanning tree $T$.}
\end{figure}

\begin{Lemma}[Beke, Cs\'aji, Csikv\'ari, Pituk \cite{beke2024permutation}] \label{conn}
Let $M$ be a matroid. For each basis $A$ of $M$, let $H[A]$ be the local basis exchange graph with respect to $A$. Then
$$T_M(x,y)=\sum_{A\in \mathcal{B}(M)}\widetilde{T}_{H[A]}(x,y),$$
where the sum runs over the set of bases $\mathcal{B}(M)$ of $M$.
\end{Lemma}

\begin{proof}
For a fixed basis $A$ and a permutation $\pi$ of the edges, the internally and externally active edges correspond to the internally and externally active vertices of $H[A]$. Hence
$$T_M(x,y)=\sum_{A\in \mathcal{B}(M)}x^{\ia_{H[A]}(\pi)}y^{\ea_{H[A]}(\pi)}.$$
Now averaging it for all permutations $\pi \in S_m$ we get that
\begin{align*}
T_M(x,y)&=\frac{1}{m!}\sum_{\pi \in S_m}T_M(x,y)\\
&=\frac{1}{m!}\sum_{\pi \in S_m}\sum_{A\in \mathcal{B}(M)}x^{\ia_{H[A]}(\pi)}y^{\ea_{H[A]}(\pi)}\\
&=\sum_{A\in \mathcal{B}(M)}\frac{1}{m!}\sum_{\pi \in S_m}x^{\ia_{H[A]}(\pi)}y^{\ea_{H[A]}(\pi)}\\
&=\sum_{A\in \mathcal{B}(M)}\widetilde{T}_{H[A]}(x,y).
\end{align*}

\end{proof}

\begin{Rem}
The local basis exchange graph $H[A]$ has an isolated vertex if and only if $M$ contains a loop or a coloop. Furthermore, in case of a graph $G$ and a spanning tree $T$, the graph $H[T]$ is connected if and only if $G$ is $2$-connected.
\end{Rem}

The following lemma is a variant of the transfer lemma \cite{beke2024permutation} with the exact same proof.  It enables us to study certain quadratic inequalities of the Tutte polynomial.

\begin{Lemma}[Transfer lemma \cite{beke2024permutation}] \label{quadratic-connection}
Let $M$ be a matroid and $x_0,x_1,x_2,y_0,y_1,y_2\geqslant0$. Suppose that 
for any basis $A$ of the matroid $M$, the local basis exchange graph $H[A]$ satisfies 
$$\widetilde{T}_{H[A]}(x_1,y_1)\widetilde{T}_{H[A]}(x_2,y_2)\geqslant \widetilde{T}_{H[A]}(x_0,y_0)^2,$$
then
$$T_M(x_1,y_1)T_M(x_2,y_2)\geqslant T_M(x_0,y_0)^2.$$
\end{Lemma}

\begin{Rem} In this paper, we aim to prove the inequality
$$\widetilde{T}_{H}(x,0)\widetilde{T}_{H}(0,x)\geqslant \widetilde{T}_{H}(1,1)^2$$
for various values of $x$. Observe that
$\widetilde{T}_{H}(1,1)=1$, a direct consequence of the definition. So what we really need to prove is 
$$\widetilde{T}_{H}(x,0)\widetilde{T}_{H}(0,x)\geqslant 1.$$

\end{Rem}

A key example for bounding the permutation Tutte polynomial is the following theorem proved in \cite{beke2024permutation}. 

\begin{Th}[\cite{beke2024permutation}] \label{lower-bound}
Let $H=(A,B,E)$ be an arbitrary bipartite graph, and let $d_i$ be the degree of a vertex $i$. Suppose that $0\leqslant x\leqslant 1$ and $y\geqslant 1$, or $0\leqslant   y\leqslant 1$ and $x\geqslant 1$. Then
$$\widetilde{T}_H(x,y)\geqslant \prod_{i\in A}\left(1+\frac{x-1}{d_i+1}\right) \cdot \prod_{j\in B}\left(1+\frac{y-1}{d_j+1}\right).$$
\end{Th}

The proof of Lemma~\ref{lower-bound} is based on the following inequality of Harris \cite{harris1960lower} that is also a special case of the FKG-inequality \cite{fortuin1971correlation}.

\begin{Lemma}[Harris \cite{harris1960lower}, Fortuin, Kasteleyn, Ginibre \cite{fortuin1971correlation}] \label{FKG-inequality}
Let $S_1\times S_2\times \dots \times S_N$ be a set, and suppose that $\mu=\mu_1\otimes \dots \otimes \mu_N$ is a product measure on it. Let $X_1,\dots ,X_t$ be non-negative monotone increasing functions in the sense that if $x_i\geqslant x_i'$ for $i=1,\dots ,N$, then for $1\leqslant   j\leqslant   t$ we have
$$X_j(x_1,\dots ,x_N)\geqslant X_j(x_1',\dots ,x_N').$$
Then
$$\E_{\mu}\left[\prod_{j=1}^tX_j\right]\geqslant \prod_{j=1}^t\E_{\mu}[X_j].$$
Moreover, let $Y$ be a monotone decreasing function, that is, for $x_i\geqslant x_i'$ $(1\leqslant i\leqslant N)$ we have
$$Y(x_1,\dots ,x_N)\leqslant Y(x_1',\dots ,x_N').$$
If $X$ is monotone increasing and $Y$ is monotone decreasing, then
$$\E_{\mu}[XY]\leqslant \E_{\mu}[X]\E_{\mu}[Y].$$
\end{Lemma}

In this paper, we will use Harris's inequality for both the product space of $[0,1]$ intervals, that is, $[0,1]^N$, and for the discrete set $\{0,1\}^A$ for some set $A$. 

In what follows, we repeatedly use the following crucial idea to express $\widetilde{T}_H(x,y)$. We can create a random ordering of the vertices of $H$ as follows: for each vertex $i$ we choose a uniform random number $x_i$ from the interval $[0,1]$. The numbers $x_i$ then determine the ordering of the vertices of $H$. The probability that two numbers are equal is $0$.

 \begin{Lemma}[\cite{beke2024permutation}]\label{vlsz}
Let $H$ be a bipartite graph and let $\widetilde{T}_H(x,y)=\sum t_{i,j}(H)x^iy^j$. Let $v(H)=m$ and let $x_1, x_2, \dots x_m$ be i.i.d. random variables with distribution $x_i\sim U(0,1)$. Let $I(A)=\left|\left\{v\in A |\  x_v\ge x_{v'} \text{ for } v'\in N_H(v)\right\}\right|$ and  $I(B)=\left|\left\{v\in B |\ x_v\ge x_{v'} \text{ for } v'\in N_H(v)\right\}\right|$. Then
$$\mathbb{P}\left(I(A)=i, I(B)=j\right)=t_{i,j}(H).$$
\end{Lemma}

For $i\in A$, let us introduce the random variable
$$Z_{i}(x_i,\{x_j\}_{j\in B})=\left\{ \begin{array}{ll}
x & \mbox{if}\ \max_{j\in N_H(i)}x_j\leqslant   x_i,\\
1 & \mbox{if}\ \max_{j\in N_H(i)}x_j> x_i.
\end{array} \right.$$
and for $j\in B$, let
$$Z_{j}(\{x_i\}_{i \in A},x_j)=\left\{ \begin{array}{ll}
y & \mbox{if}\ \max_{i\in N_H(j)}x_i\leqslant   x_j,\\
1 & \mbox{if}\ \max_{i\in N_H(j)}x_i>x_j.
\end{array} \right.$$
As a consequence of Lemma~\ref{vlsz} we have
\begin{equation} \label{eq-vlsz}
\wT_H(x,y)=\mathbb{E}\left[\prod_{v\in V(H)}Z_v(\underline{x})\right]=\int_{[0,1]^{V(H)}}.\prod_{v\in V(H)}Z_v(\underline{x})\ d\underline{x}.
\end{equation}
The problem with the random variables $Z_v(\underline{x})$ is that they are neither monotone increasing, nor monotone decreasing.

In what follows, we do a little trick. For $i\in A$ we generate $x_i\sim U(0,1)$ as before, but for $j\in B$  we first generate a uniform random number $y_j$ from $[0,1]$ and let $x_j=1-y_j$. 
The role of this trick will become apparent soon.
 
For $i\in A$, let us introduce the random variable
$$X_{i}(x_i,\{y_j\}_{j\in B})=\left\{ \begin{array}{ll}
x & \mbox{if}\ \max_{j\in N_H(i)}(1-y_j)\leqslant   x_i,\\
1 & \mbox{if}\ \max_{j\in N_H(i)}(1-y_j)> x_i.
\end{array} \right.$$
and for $j\in B$, let
$$Y_{j}(\{x_i\}_{i \in A},y_j)=\left\{ \begin{array}{ll}
y & \mbox{if}\ \max_{i\in N_H(j)}x_i\leqslant   1-y_j,\\
1 & \mbox{if}\ \max_{i\in N_H(j)}x_i>1-y_j.
\end{array} \right.$$

\begin{Lemma}
(a) We have 
$$\widetilde{T}_H(x,y)=\E\left[ \prod_{i\in A}X_i\cdot \prod_{j\in B}Y_j\right].$$
(b1) If $x\geqslant1$, then $X_{i}(x_i,\{y_j\}_{j\in B})$ is a monotone increasing function for each $i\in A$. \\
(b2) If $0\leqslant   x\leqslant   1$, then $X_{i}(x_i,\{y_j\}_{j\in B})$ is a monotone decreasing function for each $i\in A$. \\
(b3) For $0\leqslant   y\leqslant   1$ the function $Y_{j}(\{x_i\}_{i \in A},y_j)$ increases monotonically for each $j\in B$.\\ 
(b4) Finally, for $y\geqslant1$ the function $Y_{j}(\{x_i\}_{i \in A},y_j)$ decreases monotonically for each $j\in B$.
\end{Lemma}

We will also need a slight extension of the above ideas, where each vertex gets its own activity.
For $i\in A$, let us introduce the random variable
$$\widehat{X}_{i}(x_i,\{y_j\}_{j\in B})=\left\{ \begin{array}{ll}
x^{(i)} & \mbox{if}\ \max_{j\in N_H(i)}(1-y_j)\leqslant   x_i,\\
1 & \mbox{if}\ \max_{j\in N_H(i)}(1-y_j)> x_i.
\end{array} \right.$$
and for $j\in B$, let
$$\widehat{Y}_{j}(\{x_i\}_{i \in A},y_j)=\left\{ \begin{array}{ll}
y^{(j)} & \mbox{if}\ \max_{i\in N_H(j)}x_i\leqslant   1-y_j,\\
1 & \mbox{if}\ \max_{i\in N_H(j)}x_i>1-y_j.
\end{array} \right.$$
The following lemma  is just a trivial extension of the previous lemma together with Harris's inequality.

\begin{Lemma} \label{extension}
(a1) If $i\in A$ and  $x^{(i)}\geqslant 1$, then $\widehat{X}_{i}(x_i,\{y_j\}_{j\in B})$ is a monotone increasing function. \\
(a2) If $i\in A$ and $0\leqslant   x^{(i)}\leqslant 1$, then $\widehat{X}_{i}(x_i,\{y_j\}_{j\in B})$ is a monotone decreasing function. \\
(a3) If $j\in B$ and $0\leqslant   y^{(j)}\leqslant 1$ the function $\widehat{Y}_{j}(\{x_i\}_{i \in A},y_j)$ is monotone increasing.\\ 
(a4) Finally, if $j\in B$ and $y^{(j)}\geqslant 1$ the function $\widehat{Y}_{j}(\{x_i\}_{i \in A},y_j)$ is monotone decreasing.\\
(b) If $x^{(i)}\geqslant1$ for all $i\in A$ and $0\leqslant   y^{(j)}\leqslant   1$ for all $j\in B$, then
$$\E\left[ \prod_{i\in A}\widehat{X}_i\cdot \prod_{j\in B}\widehat{Y}_j\right]\geqslant \prod_{i\in A}\E[\widehat{X}_i]\cdot \prod_{j\in B}\E [\widehat{Y}_j]=\prod_{i\in A}\left(\frac{x^{(i)}}{d_i+1}+\frac{d_i}{d_i+1}\right)\cdot \prod_{j\in B}\left(\frac{y^{(j)}}{d_j+1}+\frac{d_j}{d_j+1}\right).$$
\end{Lemma}

We will use one more lemma from the paper \cite{beke2024permutation}, namely the gluing lemma. Originally, this lemma used the condition that $H_1$ and $H_2$ are trees, but the proof never used this condition.

\begin{Lemma}[Gluing lemma \cite{beke2024permutation}] \label{P(H) of glued trees}
Let $x\geqslant1$ and $0\leqslant   y\leqslant   1$.
Let $H_1$ and $H_2$ be rooted bipartite graphs with root vertices $v_1$ and $v_2$, respectively. Let $H$ be obtained from $H_1$ and $H_2$ by identifying $v_1$ and $v_2$ in the union of $H_1$ and $H_2$. Let $v$ be the vertex obtained from identifying $v_1$ and $v_2$. Assume that the bipartite parts of $H$ determines the bipartite parts of $H_1$ and $H_2$, that is, if $v\in A(H)$, then $v_1\in A(H_1)$ and $v_2\in A(H_2)$, and if $v\in B(H)$, then $v_1\in B(H_1)$ and $v_2\in B(H_2)$.
\medskip

\noindent (a) If $v\in A$, then
$$x\widetilde{T}_H(x,y)\geqslant \widetilde{T}_{H_1}(x,y)\widetilde{T}_{H_2}(x,y).$$
\noindent (b) If $v\in B$, then
$$\widetilde{T}_H(x,y)\geqslant \widetilde{T}_{H_1}(x,y)\widetilde{T}_{H_2}(x,y).$$
\noindent (c) In particular,
$$\wT_H(x,0)\wT_H(0,x)\geqslant \frac{1}{x}(\wT_{H_1}(x,0)\wT_{H_1}(0,x))(\wT_{H_2}(x,0)\wT_{H_2}(0,x)).$$
\end{Lemma}

\section{Warm-up: asymptotic computation of some permutation Tutte polynomials}
\label{sect: warm-up}

In this section, we motivate an important technique of this paper by computing the asymptotic values of the permutation Tutte polynomial of certain bipartite graphs.

Let us start with a very simple example, the complete bipartite graph.

\begin{Th}
We have
$$\wT_{K_{a,b}}(x,0)=ab\int_0^1\int_s^1s^{b-1}(s+x(t-s))^{a-1}\, dt \, ds.$$
\end{Th}

\begin{proof}
Let $K_{a,b}=(A,B,E)$ with $|A|=a,|B|=b$. As before we generate the random permutation on $A\cup B$ by first generating an $x_v\in (0,1)$ uniformly at random for all $v\in V(K_{a,b})$, and then we take the relative order of $x_v$'s. Let $t=\max_{v\in A}x_v$ and $s=\max_{v\in B}x_v$. Recall that from equation~\ref{eq-vlsz} we have
\begin{equation*} 
\wT_H(x,y)=\mathbb{E}\left[\prod_{v\in V(H)}Z_v(\underline{x})\right]=\int_{[0,1]^{V(H)}}.\prod_{v\in V(H)}Z_v(\underline{x})\ d\underline{x}.
\end{equation*}
We will refer to $\prod_{v\in V(H)}Z_v(\underline{x})$ as the weight corresponding to the permutation determined by $\underline{x}$ as it only depends on the ordering of $(x_v)_{v\in V}$.
Our plan is to compress this integral using $s$ and $t$. We have $ab$ choices for the vertices that take the values $t$ and $s$, let these vertices be $v_A$ and $v_B$. If $s>t$, then the permutation has weight $0$ as $v_B$ is an active vertex. If $t>s$, then no vertex can be active in $B$. For each $w\in A$ that is not equal to $v_A$ two things can happen: if $x_w<s$, then $w$ is not active, so $Z_w(\underline{x})=1$, or $s\leqslant x_w\leqslant t$ and then $w$ is active so $Z_w(\underline{x})=x$. This shows that as $x_w$ runs over the interval $[0,t]$ its contribution to the integral is $s+x(t-s)$. For a vertex $w\in B$ we simply need to have $x_w<s$ and $w$ will not be active, so we integrate $1$ on the whole interval $[0,s]$. The integral formula then follows.
\end{proof}

\begin{Rem} Suppose that $a=\alpha m$ and $b=\beta m$, where $\alpha,\beta$ are fixed such that $\alpha+\beta=1$ and $m\to \infty$. Then the exponential growth constant of $\wT_{K_{a,b}}(x,0)$ is simply
$$\lim_{m\to \infty}\wT_{K_{a,b}}(x,0)^{1/m}=\max_{s,t}s^{\beta}(s+x(t-s))^{\alpha}.$$
Clearly, at the maximum we have $t=1$ and we simply need to maximize $s^{\beta}(s+x(1-s))^{\alpha}$. 
This turns out to be at $s=\min\left(1,\frac{\beta x}{x-1}\right)$. 
If $\beta<\frac{x-1}{x}$, then the growth constant is
$$\left(\frac{\beta x}{x-1}\right)^{\beta}\left(x+(1-x)\frac{\beta x}{x-1}\right)^{\alpha}=\alpha^{\alpha}(1-\alpha)^{1-\alpha}\frac{x}{(x-1)^{1-\alpha}}.$$
If $\beta\geqslant\frac{x-1}{x}$, then the exponential growth constant is simply $1$. 
\end{Rem}

Let us consider the graph $H_{a,b,c}$ introduced in \cite{beke2024permutation}: we start with a complete bipartite graph $K_{a,b}$ with vertex set $A\cup B$, and then attach $c$ pendant leaves to $c$ distinct vertices of $B$, let $C$ be the set of these leaf vertices. So the resulting bipartite graph has $a+c$ vertices on one side and $b$ vertices on the other side. The graphs $H_{n,n,n}$ played an important role in the refutation of the matroidal version of the Merino--Welsh conjecture. 

\begin{figure}[h!]
\begin{tikzpicture}[scale=1.2]
\node[vertex] (a1) at (1,0) [circle,fill=black] {};
\node[vertex] (a2) at (2,0) [circle,fill=black] {};
\node[vertex] (a3) at (3,0) [circle,fill=black] {};
\node[vertex] (a4) at (4,0) [circle,fill=black] {};
\node[vertex] (a5) at (5,0) [circle,fill=black] {};
\node[vertex] (a6) at (6,0) [circle,fill=black] {};
\node[vertex] (b1) at (1,1.5) [circle,fill=black] {};
\node[vertex] (b2) at (2,1.5) [circle,fill=black] {};
\node[vertex] (b3) at (3,1.5) [circle,fill=black] {};
\node[vertex] (b4) at (4,1.5) [circle,fill=black] {};
\node[vertex] (b5) at (5,1.5) [circle,fill=black] {};
\node[vertex] (b6) at (6,1.5) [circle,fill=black] {};
\node[vertex] (c1) at (1,3) [circle,fill=black] {};
\node[vertex] (c2) at (2,3) [circle,fill=black] {};
\node[vertex] (c3) at (3,3) [circle,fill=black] {};
\node[vertex] (c4) at (4,3) [circle,fill=black] {};
\node[vertex] (c5) at (5,3) [circle,fill=black] {};
\node[vertex] (c6) at (6,3) [circle,fill=black] {};
\draw (a1) -- (b1);
\draw (a1) -- (b2);
\draw (a1) -- (b3);
\draw (a1) -- (b4);
\draw (a1) -- (b5);
\draw (a1) -- (b6);
\draw (a2) -- (b1);
\draw (a2) -- (b2);
\draw (a2) -- (b3);
\draw (a2) -- (b4);
\draw (a2) -- (b5);
\draw (a2) -- (b6);
\draw (a3) -- (b1);
\draw (a3) -- (b2);
\draw (a3) -- (b3);
\draw (a3) -- (b4);
\draw (a3) -- (b5);
\draw (a3) -- (b6);
\draw (a4) -- (b1);
\draw (a4) -- (b2);
\draw (a4) -- (b3);
\draw (a4) -- (b4);
\draw (a4) -- (b5);
\draw (a4) -- (b6);
\draw (a5) -- (b1);
\draw (a5) -- (b2);
\draw (a5) -- (b3);
\draw (a5) -- (b4);
\draw (a5) -- (b5);
\draw (a5) -- (b6);
\draw (a6) -- (b1);
\draw (a6) -- (b2);
\draw (a6) -- (b3);
\draw (a6) -- (b4);
\draw (a6) -- (b5);
\draw (a6) -- (b6);
\draw (c1) -- (b1);
\draw (c2) -- (b2);
\draw (c3) -- (b3);
\draw (c4) -- (b4);
\draw (c5) -- (b5);
\draw (c6) -- (b6);
\end{tikzpicture}
\caption{The graph $H_{6,6,6}$.}
\end{figure}

It turns out that for even $n$ there are matroids for which all local basis exchange graphs are isomorphic to $H_{n,n,n}$. Indeed, let us first consider the uniform matroid $U_{n,\frac{3}{2}n}$ of rank $n$ on $\frac{3}{2}n$ elements. Then let us replace all elements of $U_{n,\frac{3}{2}n}$ with $2$ parallel elements. Let us denote the resulting matroid by $U^{(2)}_{n,\frac{3}{2}n}$. (In general, if we replace each element of a matroid $M$ with $k$ parallel elements, then we denote by $M^{(k)}$ the obtained matroid.) The resulting matroid has $3n$ elements and rank $n$. If $B$ is a basis of the matroid $U^{(2)}_{n,\frac{3}{2}n}$, then $C$ will be the set of corresponding parallel elements, and $A$ will be the remaining elements. So the local basis exchange graph will be the dual bipartite graph $H^*_{n,n,n}=(B,A\cup C,E)$. For the dual matroid of $U^{(2)}_{n,\frac{3}{2}n}$ the local basis exchange graph will be exactly $H_{n,n,n}=(A\cup C,B,E)$. Below we sketch an alternative proof for the fact that
$$\wT_{H_{n,n,n}}(2,0)\wT_{H_{n,n,n}}(0,2)<\wT_{H_{n,n,n}}(1,1)^2$$
for large enough $n$.

\begin{Th} We have
$$\lim_{n\to \infty}\wT_{H_{n,n,n}}(x,0)^{1/n}=\max_{s\in [0,1]}\left(s+x(1-s)\right)\left(xs+(1-x)\frac{s^2}{2}\right)=\begin{cases}\frac{1}{3\sqrt{3}}\frac{x^3}{x-1} &\text{if}\ \ x\geqslant\sqrt{3}, \\ 
\frac{1}{2}(x+1) &\text{if}\ \ 1< x\leqslant   \sqrt{3},\end{cases}$$
and 
$$\lim_{n\to \infty}\wT_{H_{n,n,n}}(0,x)^{1/n}=\max_{t\in [0,1]} t\left(\frac{t^2}{2}+\left(\frac{1}{2}-\frac{t^2}{2}\right)x\right)=\begin{cases}\frac{1}{3\sqrt{3}}\frac{x^{3/2}}{(x-1)^{1/2}} &\text{if}\ \ x\geqslant\frac{3}{2},  \\ 
\frac{1}{2} &\text{if}\ \ 1< x\leqslant   \frac{3}{2}.
\end{cases}$$
In particular, if $x\geqslant\sqrt{3}$ we have
$$\lim_{n\to \infty}\left(\wT_{H_{n,n,n}}(x,0)\wT_{H_{n,n,n}}(0,x)\right)^{1/n}=\left(\frac{x^3}{9(x-1)}\right)^{3/2}.$$
\end{Th}

\begin{proof}[Sketch of the proof.]
As before we generate the random permutation on \\ $V(H)=(A\cup C)\cup B$ by first generating an $x_v\in (0,1)$ uniformly at random for all $v\in V$, and then we take the relative order of $x_v$'s. Let
$$t=\max_{v\in A}x_v\ \ \ \text{and}\ \ \ s=\max_{v\in B}x_v.$$
Let $v_A$ and $v_B$ be the vertices, where these maximums are achieved.
(Note that $t$ is only maximum on $A$ and not on $A\cup C$.) As before we use equation~\ref{eq-vlsz}:
\begin{equation*} 
\wT_H(x,y)=\mathbb{E}\left[\prod_{v\in V(H)}Z_v(\underline{x})\right]=\int_{[0,1]^{V(H)}}.\prod_{v\in V(H)}Z_v(\underline{x})\ d\underline{x}.
\end{equation*}
and we aim to compress this integral using $s$ and $t$.

If $s<t$, then the contribution of this case to $\wT_{H_{n,n,n}}(x,0)$ is
\begin{equation} \label{int1}
I_1:=n^2\int_0^1\int_s^1(s+x(t-s))^{n-1}\left(\int_{0}^s(s_i+x(1-s_i))\, ds_i\right)^{n-1}(s+x(1-s))\, dt ds.
\end{equation}
Here $s+x(t-s)$ is the contribution of a vertex $w\in A$ that is not $v_A$. If $x_w=s_i<s$ for some $w\in B$ that is not $v_B$, then the attached leaf in $C$ has contribution $s_i+x(1-s_i)$. (Note that attached leaf vertex $u$ can have a value bigger than $t$ as $t$ was only a maximum among the vertices of $A$.) The last term $s+x(1-s)$ is simply the contribution of the leaf attached to $v_B$. 

If $t<s$, then the contribution of this case to $\wT_{H_{n,n,n}}(x,0)$ is
\begin{equation} \label{int2}
I_2:=n^2\int_0^1\int_s^1t^{n-1}\left(\int_{0}^t(s_i+x(1-s_i))\, ds_i+\int_{t}^sx(1-s_i)\, ds_i\right)^{n-1}x(1-s)\, dtds.
\end{equation} 
Note that we can assume that none of the vertices of $B$ are active since the contribution of such a term would be $0$.
Here, the contribution of a vertex $w\in A\setminus \{v_A\}$ is simply $t$ as it can never be active since $t<s$. If we have a vertex $w\in B$ that is not $v_B$ with $x_w=s_i<s$, then we need to distinguish two cases. If $s_i<t$, then $w$ is not active, and so the attached leaf has a contribution $s_i+x(1-s_i)$. If $s_i>t$ then, in order to make $w$ inactive the attached leaf $u$ should have a value $x_u>s_i$, but then it is automatically active and its contribution is $x(1-s_i)$. The contribution of the leaf attached to $v_B$ is $x(1-s)$ as it has to be active to make $v_B$ inactive.

We have
\begin{align*}
\lim_{n\to \infty}I_1^{1/n}&=\max_{0\leqslant s\leqslant   t\leqslant 1}(s+x(t-s))\left(xs+(1-x)\frac{s^2}{2}\right)\\
&=\max_{0\leqslant s\leqslant 1}(s+x(1-s))\left(xs+(1-x)\frac{s^2}{2}\right)
\end{align*}
The function $(s+x(1-s))\left(xs+(1-x)\frac{s^2}{2}\right)$ has either a maximum at $\left(1-\frac{1}{\sqrt{3}}\right)\frac{x}{x-1}$ if this is less than $1$, or at $s=1$. In the first case, its value is $\frac{1}{3\sqrt{3}}\frac{x^3}{x-1}$. In the second case,  its value is $\frac{1}{2}(x+1)$.

On the other hand,
$$\lim_{n\to \infty}I_2^{1/n}=\max_{0\leqslant   t\leqslant   s\leqslant   1}t\left((1-x)\frac{t^2}{2}+xt+x(s-t)-x\left(\frac{s^2}{2}-\frac{t^2}{2}\right)\right)=
\max_{0\leqslant   t\leqslant   s\leqslant   1} t\left(xs-x\frac{s^2}{2}+\frac{t^2}{2}\right)$$
Observe that at $s=t=1$ its value is $\frac{1}{2}(x+1)$. On the other hand, this is the maximum as  dropping the condition $t\leqslant s$ would lead to the following chain of inequalities
\begin{align*}
\lim_{n\to \infty}I_2^{1/n}&=\max_{0\leqslant t\leqslant s\leqslant 1} t\left(xs-x\frac{s^2}{2}+\frac{t^2}{2}\right)\\
&\leqslant \max_{0\leqslant t, s\leqslant 1} t\left(xs-x\frac{s^2}{2}+\frac{t^2}{2}\right)\\
&=\max_{0\leqslant s\leqslant 1}\left(x\left(s-\frac{s^2}{2}\right)+\frac{1}{2}\right)\\
&=\frac{1}{2}(x+1).
\end{align*}
So $\lim_{n\to \infty}I_2^{1/n}=\frac{1}{2}(x+1)$ which is always at most $\lim_{n\to \infty}I_1^{1/n}$. Then we have
\begin{align*}
\lim_{n\to \infty}\wT_{H_{n,n,n}}(x,0)^{1/n}&=\lim_{n\to \infty}(I_1+I_2)^{1/n}\\
&=\max\left(\lim_{n\to \infty}I_1^{1/n},\lim_{n\to \infty}I_2^{1/n}\right)\\
&=\lim_{n\to \infty}I_1^{1/n}.
\end{align*}
This proves the first statement. 

The statement concerning $\wT_{H_{n,n,n}}(0,x)$ is actually simpler. 
As before, let
$$t=\max_{v\in A}x_v\ \ \ \text{and}\ \ \ s=\max_{v\in B}x_v,$$
and let $v_A$ and $v_B$ be the vertices, where these maximums are achieved. Observe that we only have to consider the case $t\leqslant s$, otherwise the vertex $v_A$ is active and $X_{v_A}=0$. If $t\leqslant s$, then for any $v\in A$ different from $v_A$ the value $x_v$ can be anything between $[0,t]$ without $v$ being active. If for a vertex $v\in B$ different from $v_B$ we have $x_v=s_i$, then its unique neighbor in $C$ cannot take bigger value than $s_i$, otherwise it becomes active and the contribution of this permutation would be $0$. If $s_i<t$, then $v$ is not active, while if $s_i\in (s,t)$, then $v$ is active.  Putting all these together, we get that
\begin{align*}
\wT_{H_{n,n,n}}(0,x)&=n^2\int_0^1\int_0^st^{n-1}\left(\int_0^ts_i\, ds_i+x\int_t^ss_i\, ds_i\right)^{n-1}s\, dtds\\
&=n^2\int_0^1\int_0^st^{n-1}\left(\frac{t^2}{2}+x\left(\frac{s^2}{2}-\frac{t^2}{2}\right)\right)^{n-1}s\, dt ds.
\end{align*}
From this we get that
\begin{align*}
\lim_{n\to \infty}\wT_{H_{n,n,n}}(0,x)^{1/n}&=\max_{0\leqslant   t\leqslant   s\leqslant   1}t\left(\frac{t^2}{2}+x\left(\frac{s^2}{2}-\frac{t^2}{2}\right)\right)\\
&=\max_{0\leqslant   t\leqslant   1}t\left(\frac{t^2}{2}+x\left(\frac{1}{2}-\frac{t^2}{2}\right)\right)
\end{align*}
This expression has a maximum either at $t=\left(\frac{x}{3(x-1)}\right)^{1/2}$ if $x\geqslant\frac{3}{2}$ or at $1$ if $x\leqslant   \frac{3}{2}$. In the first case, its value is $\frac{1}{3\sqrt{3}}\frac{x^{3/2}}{(x-1)^{1/2}}$. In the second case, its value is $\frac{1}{2}$. Since $\wT_H(1,1)=1$ for any bipartite graph $H$ by definition, this completes the sketch of the proof.
\end{proof}

\begin{Rem}
There are two takeaways from this section. First, sometimes it is possible to compute the exponential growth of the permutation Tutte polynomial without explicitly computing it. Second, it is often possible to understand those permutations that contribute the most weight to the permutation Tutte polynomial. Concentrating on these permutations might give a better lower or upper bound on the permutation Tutte polynomial.
\end{Rem}

\section{The main lemma}
\label{sect: main_lemma}

In this section, we give the main ingredient of the proofs of Theorems~\ref{main theorem} and \ref{circuit length}.

\begin{Th} \label{main_lemma}
Let $H=(A,B,E)$ be a bipartite graph. For $x\geqslant2$ and a fixed $0\leqslant   s\leqslant   1$ let us introduce the function
$$\gamma_{x,s}(d)=\frac{(d+x)s}{(d+x)s+(d+1)x(1-s)}.$$
Then
$$\widetilde{T}_H(x,0)\geqslant \prod_{v\in A}\left(\frac{(d_v+x)s}{d_v+1}+x(1-s)\right)\cdot \prod_{u\in B}\left(s-\frac{s}{d_u+1}\prod_{v\in N(u)}\gamma_{x,s}(d_v)\right).$$
\end{Th}

\begin{proof}
Recall that
$$\widetilde{T}_H(x,0)=\int_{[0,1]^A}\int_{[0,1]^B}\prod_{v\in A}X_v(\underline{x},\underline{y})\cdot \prod_{u\in B}Y_u(\underline{x},\underline{y}) d\underline{x}d\underline{y}.$$
The first idea is that we only consider those $((x_v)_{v\in A},(y_u)_{u\in B})$ for which $x_u=1-y_u\leqslant   s$ for $u\in B$. Hence
$$\widetilde{T}_H(x,0)\geqslant \int_{[0,1]^A}\int_{[1-s,1]^B}\prod_{v\in A}X_v(\underline{x},\underline{y})\cdot \prod_{u\in B}Y_u(\underline{x},\underline{y}) d\underline{x}d\underline{y}.$$
For each $v\in A$ we decompose the integral over $[0,1]$ to $[0,s]$ and $[s,1]$, this way we decomposed the integral $[0,1]^A$ to the sum of $2^{|A|}$ integrals. Since $x_u\leqslant   s$ for $u\in B$ we immediately get that if $x_v\geqslant s$, then $v$ is active and so $X_v(\underline{x},\underline{y})=x$. 
Let $T\subseteq A$ be the set of those vertices $v\in A$ for which $x_v\geqslant s$. Then
\begin{align*}
&\int_{[0,1]^A}\int_{[1-s,1]^B}\prod_{v\in A}X_v(\underline{x},\underline{y})\cdot \prod_{u\in B}Y_u(\underline{x},\underline{y}) d\underline{x}d\underline{y}\\
&=\sum_{T\subseteq A}(x(1-s))^{|T|}\int_{[0,s]^{A\setminus T}}\int_{[1-s,1]^B}\prod_{v\in A\setminus T}X_v(\underline{x},\underline{y})\cdot \prod_{u\in B}Y_u(\underline{x},\underline{y}) d\underline{x}d\underline{y}.
\end{align*}
Let $N(T)\subseteq B$ be the neighbors of $T$. If $u\in N(T)$, then  the vertex $u$ cannot be active anymore, and so let us define $\widehat{Y}_u(\underline{x},\underline{y}):=Y_u(\underline{x},\underline{y})=1$. Hence
$$I_T:=\int_{[0,s]^{A\setminus T}}\int_{[1-s,1]^B}\prod_{v\in A\setminus T}X_v(\underline{x},\underline{y})\cdot \prod_{u\in B}Y_u(\underline{x},\underline{y}) d\underline{x}d\underline{y}$$
corresponds to $s^{|A\setminus T|+|B|}$ times a permutation Tutte polynomial on $(A\setminus T,B)$, where an active vertex $i\in A\setminus T$ gets a weight $x^{(i)}=x$ while in $B$ an active vertex $j$ gets a weight $y^{(j)}=1$ if it is $N(T)$, and $y^{(j)}=0$ if it is $B\setminus N(T)$, so for this term we have a lower bound from Lemma~\ref{extension}:
\begin{align*}
I_T&=s^{|A\setminus T|+|B|}\E\left[\prod_{i\in A\setminus T}\widehat{X}_i\cdot \prod_{j\in B}\widehat{Y}_j\right]\\
&\geqslant s^{|A\setminus T|+|B|}\prod_{i\in A}\E[\widehat{X}_i]\cdot \prod_{j\in B}\E [\widehat{Y}_j]\\
&=s^{|A\setminus T|+|B|}\prod_{v \in A\setminus T}\left(1+\frac{x-1}{d_v+1}\right)\prod_{u\in B\setminus N(T)}\left(1-\frac{1}{d_u+1}\right).
\end{align*}
So far we got that
\begin{align} 
\widetilde{T}_H(x,0)&\geqslant \sum_{T\subseteq A}(x(1-s))^{|T|}I_T \nonumber\\
&\geqslant \sum_{T\subseteq A}(x(1-s))^{|T|}\cdot s^{|A\setminus T|+|B|}\prod_{v \in A\setminus T}\left(1+\frac{x-1}{d_v+1}\right)\prod_{u\in B\setminus N(T)}\left(1-\frac{1}{d_u+1}\right). \nonumber\\
&= \sum_{T\subseteq A}(x(1-s))^{|T|}\cdot \prod_{v\in A\setminus T}\left(\frac{d_v}{d_v+1}s+\frac{x}{d_v+1}s\right)\cdot  \prod_{u\in B\setminus N(T)} \frac{d_us}{d_u+1}\cdot \prod_{u\in N(T)}s \label{lower_bound_2}
\end{align}
Next, let 
$$Z=\prod_{v\in A}\left(\frac{d_v}{d_v+1}s+\frac{x}{d_v+1}s+x(1-s)\right),$$
and let us consider the measure $\mu$ on $\{0,1\}^A$ for which
$$\mu(T)=\frac{1}{Z}(x(1-s))^{|T|}\prod_{A\setminus T}\left(\frac{d_v}{d_v+1}s+\frac{x}{d_v+1}s\right).$$
Note that $\mu$ is a product measure $\mu=\bigotimes_{v\in A} \mu_v$, where
$$\mu_v(0)=\frac{\frac{d_v}{d_v+1}s+\frac{x}{d_v+1}s}{\frac{d_v}{d_v+1}s+\frac{x}{d_v+1}s+x(1-s)}\ \ \ \text{and}\ \ \ \mu_v(1)=\frac{x(1-s)}{\frac{d_v}{d_v+1}s+\frac{x}{d_v+1}s+x(1-s)}.$$
For each $u\in B$ consider the function $f_u: \{0,1\}^A\to \mathbb{R}$ defined as
$$f_u(T)=\begin{cases} s & \text{if } u\in N(T), \\
\frac{d_u}{d_u+1}s  & \text{if } u\not\in N(T).
\end{cases}$$
Then we can rewrite inequality \eqref{lower_bound_2} as follows:
\begin{align*}
\widetilde{T}_H(x,0)&\geqslant \sum_{T\subseteq A}(x(1-s))^{|T|}\cdot \prod_{v\in A\setminus T}\left(\frac{d_v}{d_v+1}s+\frac{x}{d_v+1}s\right)\cdot  \prod_{u\in B\setminus N(T)} \frac{d_us}{d_u+1}\cdot \prod_{u\in N(T)}s\\
&=Z\sum_{T\subseteq A}\mu(T)\prod_{u\in B}f_u(T).
\end{align*}
All functions $f_u$ are monotone increasing on the space $\{0,1\}^{A}$ identified  with the subsets  $T$ of $A$,  since adding vertices to $T$ can only increase $N(T)$, so by Harris's inequality (Lemma~\ref{FKG-inequality}) we have
$$Z\sum_{T\subseteq A}\mu(T)\prod_{u\in B}f_u(T)\geqslant Z\prod_{u\in B}\left(\sum_{T\subseteq A}\mu(T)f_u(T)\right).$$
Here
\begin{align*}
&\sum_{T\subseteq A}\mu(T)f_u(T)=
\sum_{T\subseteq A \atop u\in N(T)}\mu(T)s+\sum_{T\subseteq A \atop u\not\in N(T)}\mu(T)s\left(1-\frac{1}{d_u+1}\right)\\
&=s-\frac{s}{d_u+1}\sum_{T\subseteq A \atop u\not\in N(T)}\mu(T)=s-\frac{s}{d_u+1}\sum_{T\subseteq A \atop T\cap N(u)=\emptyset}\mu(T)\\
&=s-\frac{s}{d_u+1}\prod_{v\in N(u)}\frac{\frac{d_v+x}{d_v+1}s}{x(1-s)+\frac{d_v+x}{d_v+1}s}=s\left(1-\frac{1}{d_u+1}\prod_{v\in N(u)}\gamma_{x,s}(d_v)\right).
\end{align*}
Hence
$$\widetilde{T}_H(x,0)\geqslant \prod_{v\in A}\left(\frac{(d_v+x)s}{d_v+1}+x(1-s)\right)\cdot \prod_{u\in B}\left(s-\frac{s}{d_u+1}\prod_{v\in N(u)}\gamma_{x,s}(d_v)\right).$$
\end{proof}

By applying Theorem~\ref{main_lemma} for both $\widetilde{T}_H(x,0)$ and $\widetilde{T}_H(0,x)$ we get the following corollary.

\begin{Cor} \label{cor: main_lemma}
We have
$$\widetilde{T}_H(x,0)\widetilde{T}_H(0,x)\geqslant \prod_{v\in V}\left[\left(\frac{(d_v+x)s}{d_v+1}+x(1-s)\right)\left(s-\frac{s}{d_v+1}\prod_{u\in N(v)}\gamma_{x,s}(d_u)\right)\right].$$
\end{Cor}

\begin{Def}
Let us introduce the notation
$$\mathbb{G}\left(d,x,s,\{\gamma_j\}_{1\leqslant j\leqslant d}\right):=\left(\frac{(d+x)s}{d+1}+x(1-s)\right)\left(s-\frac{s}{d+1}\prod_{j=1}^d\gamma_{j}\right).$$
Furthermore, let
$$G(d,x,s,\gamma):=\left(\frac{(d+x)s}{d+1}+x(1-s)\right)\left(s-\frac{s}{d+1}\gamma^d\right).$$
For any $0\leqslant   \gamma<1$ we have
$$G(\infty,x,s):=\lim_{d\to \infty}G(d,x,s,\gamma)=(s+x(1-s))s.$$
\end{Def}

Clearly, with this notation we can rewrite Corollary~\ref{cor: main_lemma} as
$$\widetilde{T}_H(x,0)\widetilde{T}_H(0,x)\geqslant \prod_{v\in V}\mathbb{G}\left(d_v,x,s,\{\gamma_{x,s}(d_u)\}_{u\in N(v)}\right).$$

Next, we prove some monotonicity properties of the function $\mathbb{G}\left(d_v,x,s,\{\gamma_{x,s}(d_u)\}_{u\in N(v)}\right)$.

\begin{Lemma} \label{monotonicity}
If $x\geqslant 1$ and $0\leqslant s\leqslant 1$, then the function $\gamma_{x,s}(d)$ is monotone decreasing in $d$. Furthermore, $\gamma_{x,s}(d)\leqslant s$.
\end{Lemma}

\begin{proof}
The function $\gamma_{x,s}(d)$ is a rational function in $d$:
$$\gamma_{x,s}(d)=\frac{(d+x)s}{(d+x)s+(d+1)x(1-s)}=\frac{sd+ sx}{(x(1-s)+s)d+x}$$
In general, the derivative of a function $f(t)=\frac{at+b}{ct+d}$ is $\frac{ad-bc}{(ct+d)^2}$, so it is enough to check that
$$sx-sx(x(1-s)+s)=sx(1-x)(1-s)\leqslant 0$$
which is clearly true. The second statement follows from $d+x\leqslant (d+1)x$ as follows:
$$\gamma_{x,s}(d)=\frac{(d+x)s}{(d+x)s+(d+1)x(1-s)}\leqslant \frac{(d+x)s}{(d+x)s+(d+x)(1-s)}=s.$$
\end{proof}

By combining Corollary~\ref{cor: main_lemma} with Lemma~\ref{monotonicity} we immediately get the following statement.

\begin{Cor} \label{cor: main_lemma_2} 
Let $H$ be a bipartite graph with minimum degree at least $\delta$. Then for $x\geqslant 2$ and $0\leqslant s\leqslant 1$ we have
$$\widetilde{T}_H(x,0)\widetilde{T}_H(0,x)\geqslant \prod_{v\in V}G(d_v,x,s,\gamma_{x,s}(\delta))\geqslant \prod_{v\in V}G(d_v,x,s,s).$$
\end{Cor}

\begin{proof}
We have
\begin{align*}
\widetilde{T}_H(x,0)\widetilde{T}_H(0,x)&\geqslant \prod_{v\in V}\mathbb{G}\left(d_v,x,s,\{\gamma_{x,s}(d_u)\}_{u\in N(v)}\right)\\
&\geqslant \prod_{v\in V}\mathbb{G}\left(d_v,x,s,\{\gamma_{x,s}(\delta)\}_{u\in N(v)}\right)\\
&=\prod_{v\in V}G(d_v,x,s,\gamma_{x,s}(\delta))\\
&\geqslant \prod_{v\in V}G(d_v,x,s,s).
\end{align*}

\end{proof}

\begin{Lemma} \label{monotonicity2}
Let $x\geqslant 1$ and $0\leqslant \gamma \leqslant s<1$. If $\left(\frac{x-1}{x(d+2)}\right)^{1/(d-1)}\geqslant \gamma$, then
$$G(d,x,s,\gamma)\geqslant G(d+1,x,s,0),$$
consequently
$$G(d,x,s,\gamma)\geqslant G(d+1,x,s,0)\geqslant G(d+1,x,s,\gamma).$$
Furthermore, the function $\left(\frac{x-1}{x(d+2)}\right)^{1/(d-1)}$ is monotone increasing in $d$ for $x\geqslant 1$ and $d\geqslant 2$, so it is enough to check the above inequality for some fixed $d_0$ to conclude that the sequence $G(d,x,s,\gamma)$ is decreasing for $d\geqslant d_0$.
\end{Lemma}

\begin{proof}
The inequality $G(d+1,x,s,0)\geqslant G(d+1,x,s,\gamma)$ is trivial. We only need to prove the inequality $G(d,x,s,\gamma)\geqslant G(d+1,x,s,0)$.
We have
\begin{small}
\begin{align*}
&G(d,x,s,\gamma)-G(d+1,x,s,0)\\
&=\left(\frac{(d+x)s}{d+1}+x(1-s)\right)\left(s-\frac{s}{d+1}\gamma^d\right)-\left(\frac{(d+1+x)s}{d+2}+x(1-s)\right)s\\
&=s^2\left(\frac{d+x}{d+1}-\frac{d+1+x}{d+2}\right)-\frac{s}{d+1}\gamma^d\left(\frac{(d+x)s}{d+1}+x(1-s)\right)\\
&= \frac{(x-1)s^2}{(d+1)(d+2)}-\frac{s}{d+1}\gamma^d\left(\frac{(d+x)s}{d+1}+x(1-s)\right)\\
&\geqslant \frac{(x-1)s^2}{(d+1)(d+2)}-\frac{sx}{d+1}\gamma^d\\
&\geqslant \frac{(x-1)s^2}{(d+1)(d+2)}-\frac{sx}{d+1}s\gamma^{d-1}\\
&=\frac{xs^2}{d+1}\left(\frac{x-1}{x(d+2)}-\gamma^{d-1}\right)
\end{align*}
\end{small}
which is non-negative if $\frac{x-1}{x(d+2)}\geqslant \gamma^{d-1}$. 
Finally, the derivative of $h(t):=\frac{1}{t-1}\ln \left(\frac{x-1}{x(t+2)}\right)$ is 
$$h'(t)=\frac{1}{(t-1)^2}\ln \left(\frac{x(t+2)}{x-1}\right)-\frac{1}{(t-1)(t+2)}.$$
Then
$$h'(t)\geqslant \frac{\ln(t+2)}{(t-1)^2}-\frac{1}{(t-1)(t+2)}\geqslant \frac{\ln(3)}{(t-1)^2}-\frac{1}{(t-1)(t+2)}>0$$
if $t>1$.
\end{proof}

\section{Proof of Theorem~\ref{main theorem}.} 
\label{sect: proof_main_theorem}

In this section, we prove that
$$T_M(x,0)T_M(0,x)\geqslant T_M(1,1)^2$$
whenever $x\geqslant 2.355$. By the transfer lemma it is enough to prove that
$$\wT_H(x,0)\wT_H(0,x)\geqslant \wT_H(1,1)^2$$
for any bipartite graph $H$ with minimum degree at least $1$. Recall that $\wT_H(1,1)=1$. 

We do not immediately prove the best value $2.355$, but give various improvements to the value $2.9243$ step by step to digest one idea at a time. All computations are carried out in the jupyter notebook \cite{JUP} and we also give the tables in the Appendix.
\bigskip

\noindent \textbf{Idea 1.} 
By Corollary~\ref{cor: main_lemma_2} we have
$$\widetilde{T}_H(x,0)\widetilde{T}_H(0,x)\geqslant \prod_{v\in V}G(d_v,x,s,\gamma_{x,s}(1)).$$
Since $\widetilde{T}_H(1,1)^2=1$ it is enough if $G(d,x,s,\gamma_{x,s}(1))\geqslant 1$ for all $d\geqslant 1$. Let
$$x=\frac{3+\sqrt{5}}{2}\approx 2.6180..\ \ \text{and}\ \ \ s=\frac{6+2\sqrt{5}}{7+3\sqrt{5}}\approx 0.7639...$$
These numbers are chosen in such a way that
$G(1,x,s,\gamma_{x,s}(1))=1$. It is easy to check that for $d_0=9$ we have 
$\left(\frac{x-1}{x(d_0+2)}\right)^{1/(d_0-1)}\geqslant s\geqslant \gamma_{x,s}(1)$. By Lemma~\ref{monotonicity} this implies that for $d\geqslant d_0$ we have
$G(d,x,s,\gamma_{x,s}(1))\geqslant G(\infty,x,s)>1$. It is also easy to check (see the jupyter notebook or Table 1 in the Appendix) that $G(d,x,s,\gamma_{x,s}(1))\geqslant 1$ for $d\in \{2,\dots ,8\}$. 
So with $x=\frac{3+\sqrt{5}}{2}$ we have 
$$\widetilde{T}_H(x,0)\widetilde{T}_H(0,x)\geqslant \widetilde{T}_H(1,1)^2$$
implying that
$$T_M(x,0)T_M(0,x)\geqslant T_M(1,1)^2$$
for all matroids without loops and coloops, in particular 
for all graphs $G$ without loops and bridges.
\bigskip

\noindent \textbf{Idea 2.}  Observe that we can assume that $H$ is connected since 
$\wT_H(x,y)=\prod_{i=1}^k\wT_{H_i}(x,y)$
if $H$ has connected components $H_1,\dots ,H_k$. 
The statement $\wT_H(2,0)\wT_H(0,2)\geqslant1$ is trivially true for $H=K_2$, the complete graph on $2$ vertices. So we can assume that $H$ has at least $3$ vertices.  On the other hand, if $H$ is connected and has at least $3$ vertices, then a degree one vertex cannot be adjacent to another degree one vertex. We claim that  it is enough to check that if
\begin{align} \label{conditions: idea 2}
G(1,x,s,\gamma_{x,s}(2))\geqslant1\ \ \text{and}\ \ G(d,x,s,\gamma_{x,s}(1))\geqslant1
\end{align}
for all $d\geqslant2$.  Indeed,
\begin{align*}
\widetilde{T}_H(x,0)\widetilde{T}_H(0,x)&\geqslant\prod_{v\in V}\mathbb{G}\left(d_v,x,s,\{\gamma_{x,s}(d_u)\}_{u\in N(v)}\right)\\
&=\prod_{v\in V: d_v=1}\mathbb{G}\left(d_v,x,s,\{\gamma_{x,s}(d_u)\}_{u\in N(v)}\right)\cdot \prod_{v\in V: d_v\geqslant2}\mathbb{G}\left(d_v,x,s,\{\gamma_{x,s}(d_u)\}_{u\in N(v)}\right)\\
&\geqslant \prod_{v\in V: d_v=1}G(1,x,s,\gamma_{x,s}(2))\cdot \prod_{v\in V: d_v\geqslant2}G(d_v,x,s,\gamma_{x,s}(1))\\
&\geqslant1.
\end{align*}
We claim that the conditions of \ref{conditions: idea 2} are satisfied if $x=2.54$ and $s=0.76$, see Table 2. Clearly, the condition $G(1,x,s,\gamma_{x,s}(2))\geqslant 1$ is just a matter of computation. To prove that $G(d,x,s,\gamma_{x,s}(1))\geqslant 1$ for all $d\geqslant d_0=9$, we use the same argument as before: for $d\geqslant d_0$ the sequence is $G(d,x,s,\gamma_{x,s}(1))$ is monotone decreasing by Lemma~\ref{monotonicity2} and checking that
$$\left(\frac{x-1}{x(d_0+2)}\right)^{1/(d_0-1)} > \gamma_{x,s}(1).$$
This gives that 
$$G(d,x,s,\gamma_{x,s}(1))\geqslant G(\infty,x,s)>1$$
for all $d\geqslant d_0$. Since $G(d,x,s,\gamma_{x,s}(1))>1$ for $2\leqslant   d\leqslant   8$ we get this inequality for all $d\geqslant 2$.

\bigskip

\noindent \textbf{Idea 3.} The next idea is to group the degree $1$ vertices together with their neighbors:
\begin{align*}
\widetilde{T}_H(x,0)\widetilde{T}_H(0,x)&\geqslant\prod_{v\in V}\mathbb{G}\left(d_v,x,s,\{\gamma_{x,s}(d_u)\}_{u\in N(v)}\right)\\
&=\prod_{v\in V: d_v\geqslant 2}\left(\mathbb{G}\left(d_v,x,s,\{\gamma_{x,s}(d_u)\}_{u\in N(v)}\right)\cdot \prod_{u\in N(v): d_u=1}\mathbb{G}\left(d_u,x,s,\{\gamma_{x,s}(d_{u'})\}_{u'\in N(u)}\right)\right)\\
&=\prod_{v\in V: d_v\geqslant 2}\left(\mathbb{G}\left(d_v,x,s,\{\gamma_{x,s}(d_u)\}_{u\in N(v)}\right)\cdot \prod_{u\in N(v): d_u=1}G\left(1,x,s,\gamma_{x,s}(d_v)\right)\right)\\
&\geqslant\prod_{v\in V: d_v\geqslant 2}\left(\mathbb{G}\left(d_v,x,s,\{\gamma_{x,s}(1)\}_{u\in N(v)}\right)\cdot \prod_{u\in N(v): d_u=1}G\left(1,x,s,\gamma_{x,s}(d_v)\right)\right)\\
&=\prod_{v\in V: d_v\geqslant 2}\left(G(d_v,x,s,\gamma_{x,s}(1))\prod_{u\in N(v): d_u=1}G\left(1,x,s,\gamma_{x,s}(d_v)\right)\right)
\end{align*}
It would be enough for us if for all $v$ with $d_v\geqslant2$ we have
$$G(d_v,x,s,\gamma_{x,s}(1))\prod_{u\in N(v): d_u=1}G\left(1,x,s,\gamma_{x,s}(d_v)\right)\geqslant 1.$$
In order to prove such an inequality, we need an upper bound for the number of degree $1$ neighbors of a vertex $v$. We claim  that if $x\geqslant2.2$, then any vertex can have at most $2$ leaf neighbors. To prove this, we use the gluing lemma. First, observe that for the star $S_4$ on $4$ vertices we have
$$\frac{1}{x}\widetilde{T}_{S_4}(x,0)\widetilde{T}_{S_4}(0,x)=\frac{1}{x}\cdot \frac{x^3+x^2+x}{4}\cdot \frac{x}{4}=\frac{x^3+x^2+x}{16}>1,$$
so by the gluing lemma (Lemma~\ref{P(H) of glued trees}) we know that a graph containing a pendant $S_4$ cannot be a minimal
counter example. This means that in a minimal counter example every vertex has at most $2$ neighbors of degree $1$. So, it is enough to check if for every $d\geqslant2$ we have
$$\min_{0\leqslant   k\leqslant   \min(2,d-1)}G(d,x,s,\gamma_{x,s}(1))G(1,x,s,\gamma_{x,s}(d))^{k}\geqslant1.$$
Here we used that if $d=2$ then it cannot have two neighbors of degree $1$, since then $K_{1,2}$ would be a connected component of $H$, but it is clearly not a counter example. We show that for $x=2.36$ and $s=0.78$ these inequalities are satisfied. We can check this inequality separately for $d=2$. For $d>2$ we have
\begin{align*}
&\ \min_{0\leqslant   k\leqslant   \min(2,d-1)}G(d,x,s,\gamma_{x,s}(1))G(1,x,s,\gamma_{x,s}(d))^{k}\\
&=\min(G(d,x,s,\gamma_{x,s}(1)), G(d,x,s,\gamma_{x,s}(1))G(1,x,s,\gamma_{x,s}(d))^{2}).
\end{align*}
(Clearly, if $G(1,x,s,\gamma_{x,s}(d))>1$ then the minimum is achieved at the first term, otherwise the minimum is achieved at the second term.)

In order to check that the above minimum is at least $1$ we proceed as before. Suppose that $G(\infty,x,s)\geqslant1,$ and for some $d_0$ we have 
\begin{align} \label{idea3-ineq-1}
G(\infty,x,s)G(1,x,s,\gamma_{x,s}(d_0))^{2}\geqslant1,
\end{align}
and for $d\geqslant d_0$ the sequence $G(d,x,s,\gamma_{x,s}(1))$ is monotone decreasing. To check this latter property, we only need to check 
$$\left(\frac{x-1}{x(d_0+2)}\right)^{1/(d_0-1)} > \gamma_{x,s}(1).$$
Then for $d\geqslant d_0$ we have
$$G(d,x,s,\gamma_{x,s}(1))G(1,x,s,\gamma_{x,s}(d))^{2}\geqslant G(\infty,x,s)G(1,x,s,\gamma_{x,s}(d_0))^{2}\geqslant1,$$
and 
$$G(d,x,s,\gamma_{x,s}(1))\geqslant G(\infty,x,s)\geqslant1.$$
With the choice $x=2.36, s=0.78$ and $d_0=44$ all these inequalities are satisfied. Then we can check the remaining inequalities for $d=3,\dots ,43$ with a computer separately. Again see the jupyter notebook \cite{JUP} or alternatively, Table 3.
\bigskip

\noindent \textbf{Idea 4.} 
We can gain a very small further improvement as follows.
We start our series of inequalities just as in the previous step:
\begin{align*}
\widetilde{T}_H(x,0)\widetilde{T}_H(0,x)&\geqslant\prod_{v\in V}\mathbb{G}\left(d_v,x,s,\{\gamma_{x,s}(d_u)\}_{u\in N(v)}\right)\\
&=\prod_{v\in V: d_v\geqslant 2}\left(\mathbb{G}\left(d_v,x,s,\{\gamma_{x,s}(d_u)\}_{u\in N(v)}\right)\cdot \prod_{u\in N(v): d_u=1}\mathbb{G}\left(d_u,x,s,\{\gamma_{x,s}(d_{u'})\}_{u'\in N(u)}\right)\right)\\
&=\prod_{v\in V: d_v\geqslant 2}\left(\mathbb{G}\left(d_v,x,s,\{\gamma_{x,s}(d_u)\}_{u\in N(v)}\right)\cdot \prod_{u\in N(v): d_u=1}G\left(1,x,s,\gamma_{x,s}(d_v)\right)\right)
\end{align*}
Previously, we bounded $\mathbb{G}\left(d_v,x,s,\{\gamma_{x,s}(d_u)\}_{u\in N(v)}\right)$ with $G(d_v,x,s,\gamma_{x,s}(1))$, but now we observe that if $d_v>2$, then $v$ can have at most two neighbors of degree $1$, while if $d_v=2$ then $v$ can have at most one vertex of degree $1$, all other neighbors have degree at least $2$. Thus, for $d_v\ge 2$ we can use that   
\begin{align*}
&\mathbb{G}\left(d_v,x,s,\{\gamma_{x,s}(d_u)\}_{u\in N(v)}\right)\\
&\geqslant\left(\frac{(d_v+x)s}{d_v+1}+x(1-s)\right)\left(s-\frac{s}{d_v+1}\gamma_{x,s}(1)^{\min(2,d_v-1)}\gamma_{x,s}(2)^{d_v-\min(2,d_v-1)}\right),
\end{align*}
Let 
$$G_k(d,x,s,\gamma_1,\gamma_2):= \left(\frac{(d+x)s}{d+1}+x(1-s)\right)\left(s-\frac{s}{d+1}\gamma_1^k\gamma_2^{d-k}\right).$$
Then it is enough to check that for all $d\geqslant2$ we have
$$\min_{0\leqslant   k\leqslant   \min(2,d-1)}G_k(d,x,s,\gamma_{x,s}(1),\gamma_{x,s}(2))\cdot G(1,x,s,\gamma_{x,s}(d))^k\geqslant1.$$
We show that this is indeed true if $x=2.355$ and $s=0.78$. First, we check it for $d\geqslant d_0=100$. By Lemma~\ref{monotonicity2} the sequence $G(d,x,s,\gamma)$ is monotone decreasing for $d\geqslant d_0$ if $\left(\frac{x-1}{x(d_0+2)}\right)^{1/(d_0-1)}\geqslant\gamma$. Thus for $0\leqslant   k\leqslant   2$ we have
\begin{align*}
G_k(d,x,s,\gamma_{x,s}(1),\gamma_{x,s}(2))\cdot G(1,x,s,\gamma_{x,s}(d))^k&\geqslant G(d,x,s,\gamma_{x,s}(1))G(1,x,s,\gamma_{x,s}(d_0))^k\\
&\geqslant G(\infty,x,s)G(1,x,s,\gamma_{x,s}(d_0))^k\\
&>1,
\end{align*}
where the last inequality can be checked directly. 
For $d<d_0$ we again check the values by a computer. See \cite{JUP} or Table 4 in the Appendix for further details. 
\bigskip

\begin{Rem} It turns out that the value $2.355$ can be further improved to $2.35$ by observing that one can disregard one particular inequality:
$$G_2(3,x,s,\gamma_{x,s}(1),\gamma_{x,s}(2))G(1,x,s,\gamma_{x,s}(3))^2\geqslant1.$$
This inequality comes from the situation where a vertex $v$  of degree $3$ is adjacent to two vertices of degree $1$. But this means that $H$ can be obtained by gluing $H_1=S_4$ to some $H_2$ at one of its leaves. As we discussed at Idea 3 this would mean that $H$ is not a minimal counter example. Choosing $x=2.35$, $s=0.79$ and $d_0=1000$ the remaining parts of Idea 4 work verbatimely.

\end{Rem}

\section{Matroids with prescribed circuit lengths}
\label{sect: circuit_length}

In this section, we prove Theorem~\ref{circuit length}.
The heart of the argument is the following lemma.

\begin{Lemma} \label{lemma: circuit length}
Let $k\geqslant4$ be a real number. Let $H$ be a bipartite graph such that all degrees of $H$ lie in the interval $[k+1,k^4-2k^2-1]$. Then
$$\wT_H(2,0)\wT_H(0,2)\geqslant\wT_H(1,1)^2.$$
\end{Lemma}

\begin{proof}
We will use that $\wT_H(1,1)=1$, and Corollary~\ref{cor: main_lemma_2}:
$$\wT_H(2,0)\wT_H(0,2)\geqslant\prod_{v\in V}G(d_v,2,s,s).$$
Note that
$$G(d,2,s,s)=\left(\frac{d+2}{d+1}\cdot s+2(1-s)\right)\left(s-\frac{s^{d+1}}{d+1}\right)=\left(\frac{d+2}{d+1}+\frac{d}{d+1}(1-s)\right)\left(s-\frac{s^{d+1}}{d+1}\right).$$
We will show that if $s=1-\frac{1}{k^2}$, then for all $d\in [k+1,k^4-2k^2-1]$ we have $G(d,2,s,s)\geqslant1$. First, observe that $G(d,2,s,s)$ is a concave function of $s$ on 
the interval $(0,1)$ as its second derivative is
$$-2\frac{d}{d+1}(1-s^d)+\left(\frac{d+2}{d+1}+\frac{d}{d+1}(1-s)\right)(-ds^{d-1})$$
which is clearly negative as both terms are negative. (Here we used the derivation rule $(h_1\cdot h_2)''=h_1''\cdot h_2+2h_1'\cdot h_2'+h_1\cdot h_2''$ for $h_1=\frac{d+2}{d+1}+\frac{d}{d+1}(1-s)$ and $h_2=s-\frac{s^{d+1}}{d+1}$ with $h_1'=-\frac{d}{d+1}$, $h_1''=0$, $h_2'=1-s^d$, $h_2''=-ds^{d-1}$.)

This means that
$$I_d:=\{s\ |\ G(d,2,s,s)\geqslant1\}$$ 
is an interval for each $d$. We will show that $1-\frac{1}{k^2}\in I_{k+1}\cap I_{k^4-2k^2-1}$. By first choosing $k=d-1$ we get that $1-\frac{1}{(d-1)^2}\in I_d$. If we choose $k^2=1+\sqrt{d+2}$, then $k^4-2k^2-1=d$ and so $1-\frac{1}{1+\sqrt{d+2}}\in I_d$. Hence $\left[1-\frac{1}{1+\sqrt{d+2}},1-\frac{1}{(d-1)^2}\right]\subseteq I_d$. This means that if $k+1\leqslant   d\leqslant   k^4-2k^2-1$, then $1-\frac{1}{k^2}\in I_d$.

First we show that $1-\frac{1}{k^2}\in I_{k+1}$, that is, $G(k+1,2,1-\frac{1}{k^2},1-\frac{1}{k^2})\geqslant1$. We will use that
$$(1-t)^r\leqslant   1-\binom{r}{1}t+\binom{r}{2}t^2$$
for $0\leqslant   t\leqslant   1$. This can be either proved  by Bonferroni's inequality applied to independent events with probability $t$, or by induction: the base case $r=1$ being trivial we have
\begin{align*}
(1-t)^r&=(1-t)^{r-1}(1-t)\\
&\leqslant \left(1-\binom{r-1}{1}t+\binom{r-1}{2}t^2\right)(1-t)\\
&=1-\binom{r}{1}t+\binom{r}{2}t^2-\binom{r-1}{2}t^2\\
&\leqslant 1-\binom{r}{1}t+\binom{r}{2}t^2.
\end{align*}
\noindent Hence
\begin{align*}
G\left(d,2,1-\frac{1}{k^2},1-\frac{1}{k^2}\right)&=
\left(\frac{d+2}{d+1}+\frac{d}{d+1}\cdot \frac{1}{k^2}\right)\left(1-\frac{1}{k^2}-\frac{1}{d+1}\left(1-\frac{1}{k^2}\right)^{d+1}\right)\\
&\geqslant\left(\frac{d+2}{d+1}+\frac{d}{d+1}\cdot \frac{1}{k^2}\right)\left(1-\frac{1}{k^2}-\frac{1}{d+1}\left(1-\frac{d+1}{k^2}+\frac{(d+1)d}{2k^4}\right)\right).
\end{align*}
For $d=k+1$ we get that
\begin{small}
\begin{align*}
G\left(k+1,2,1-\frac{1}{k^2},1-\frac{1}{k^2}\right)&\geqslant\left(\frac{k+3}{k+2}+\frac{k+1}{k+2}\cdot \frac{1}{k^2}\right)\left(1-\frac{1}{k^2}-\frac{1}{k+2}\left(1-\frac{k+2}{k^2}+\frac{(k+2)(k+1)}{2k^4}\right)\right)\\
&=1+\frac{1}{2}\frac{3k^5 - 4k^4 - 12k^3 - 10k^2 - 5k - 2}{k^8 + 4k^7 + 4k^6}.
\end{align*}
\end{small}
For $k\geqslant4$ we have $3k^5 - 4k^4 - 12k^3 - 10k^2 - 5k - 2>0$ showing that $G\left(k+1,2,1-\frac{1}{k^2},1-\frac{1}{k^2}\right)>1$.

Next we show that $1-\frac{1}{k^2}\in I_{k^4-2k^2-1}$. For $s=1-\frac{1}{k^2}$ and $d=k^4-2k^2-1$ we have 
$$\left(\frac{d+2}{d+1}+\frac{d}{d+1}(1-s)\right)s=1+\frac{1}{k^8-2k^6}$$
and
$$\left(\frac{d+2}{d+1}+\frac{d}{d+1}(1-s)\right)\frac{s^{d+1}}{d+1}=\frac{k^6-k^4-k^2-1}{k^6-2k^4}\cdot \frac{1}{k^4-2k^2}\left(1-\frac{1}{k^2}\right)^{k^4-2k^2}.$$
So we need to prove that
$$\frac{1}{k^8-2k^6}>\frac{k^6-k^4-k^2-1}{k^6-2k^4}\cdot \frac{1}{k^4-2k^2}\left(1-\frac{1}{k^2}\right)^{k^4-2k^2}$$
which is equivalent with
$$\frac{k^2-2}{k^6-k^4-k^2-1}>\left(1-\frac{1}{k^2}\right)^{k^4-2k^2}.$$
By $1-t<e^{-t}$ we have
$$\left(1-\frac{1}{k^2}\right)^{k^4-2k^2}<e^{-\frac{1}{k^2}(k^4-2k^2)}=e^{-k^2+2}.$$
Let us introduce $y=k^2-2$. Then $k^6-k^4-k^2-1=y^3+5y^2+7y+1$, so it is enough to prove that
$$\frac{y}{y^3+5y^2+7y+1}>e^{-y}$$
or equivalently
$$y^2+5y+7+\frac{1}{y}<e^y.$$
Since $k\geqslant3$ we have $y=k^2-2\geqslant7$, but $e^y>y^2+5y+7+\frac{1}{y}$ is already true for $y\geqslant 4$.

\end{proof}

Now we are ready to finish the proof of Theorem~\ref{circuit length}.

\begin{proof}[Proof of Theorem~\ref{circuit length}] Let $A$ be a basis of the matroid $M$ on ground set $E$. Let $B=E\setminus A$. Let us consider the local basis exchange graph $H[A]$ on $A\cup B$. Observe that $v\in B$ are adjacent with exactly those elements that are in its fundamental circuit. So $\ell\leqslant   d_v+1\leqslant   (\ell-2)^2(\ell^2-4\ell+2)$. Similarly, if $u\in A$, then it is adjacent with those elements of $B$ that are in its fundamental circuit with respect to $B$ in $M^*$. So $\ell\leqslant   d_u+1\leqslant  (\ell-2)^2(\ell^2-4\ell+2)$.

We treat the cases $\ell=4$ and $5$ separately. If $\ell=4$, then all degree $d_v$ satisfy that $3\leqslant   d_v\leqslant   7$. Let us choose $s=0.9226$. Then 
$$G(d,2,s,\gamma_{2,s}(3))>1$$
for $3\leqslant   d\leqslant   141$. By Corollary~\ref{cor: main_lemma_2} we have
$$\wT_H(2,0)\wT_H(0,2)\geqslant\prod_{v\in V}G(d_v,x,s,\gamma_{x,s}(\delta))\geqslant\prod_{v\in V}G(d_v,x,s,\gamma_{x,s}(3))>1=\wT_H(1,1)^2.$$
Together with the transfer lemma this implies the claim. In case of $\ell=5$ all degree $d_v$ satisfy $4\leqslant   d_v \leqslant   62$. With the choice $s=0.9622$ we have  
$$G(d,2,s,\gamma_{2,s}(4))>1$$
for $4\leqslant   d\leqslant   646$. Hence again by Corollary~\ref{cor: main_lemma_2} we have
$$\wT_H(2,0)\wT_H(0,2)\geqslant\prod_{v\in V}G(d_v,x,s,\gamma_{x,s}(\delta))\geqslant\prod_{v\in V}G(d_v,x,s,\gamma_{x,s}(4))>1=\wT_H(1,1)^2$$
and the transfer lemma this implies the claim.

Now let us assume that $\ell \geqslant 6$. By choosing $k=\ell-2$ we get that $k\geq 4$ and for all vertex $v$ we have $k+1=\ell-1\leq d_v$ and
$$k^4-2k^2-1=(\ell-2)^4-2(\ell-2)^2-1 \geqslant (d_v+1)-1=d_v.$$
 This means that all degree $d_v$ satisfy $k+1\leqslant   d_v\leqslant   k^4-2k^2-1$. So by Lemma~\ref{lemma: circuit length} we have 
$$\wT_{H[A]}(2,0)\wT_{H[A]}(0,2)\geqslant\wT_{H[A]}(1,1)^2.$$
Since this is true for all basis $A$ of the matroid $M$, we get by the Transfer lemma (Lemma~\ref{quadratic-connection}) that
$$T_M(2,0)T_M(0,2)\geqslant T_M(1,1)^2.$$

\end{proof}

\section{Concluding remarks}
\label{sect: concluding_remarks}

We end this paper with some conjectures. Let us start with an admittedly provocative conjecture.

\begin{Conj} \label{provocation}
Let $x_0$ be the largest zero of the polynomial $x^3-9(x-1)$. Then for any matroid $M$ without loops and coloops we have
$$T_M(x_0,0)T_M(0,x_0)\geq T_M(1,1)^2.$$
\end{Conj}

We believe that even a disproof of this conjecture would lead to a better understanding of this problem than slightly improving the constant $2.355$ (or $2.35$).

Concerning the Merino-Welsh conjecture, the following variant might be true.

\begin{Conj} \label{simple-cosimple}
If $M$ is a simple co-simple matroid, then
$$T_M(2,0)T_M(0,2)\geqslant T_M(1,1)^2.$$
\end{Conj}

The following line of attack seems to be natural.

\begin{Conj} \label{degree-2}
If the bipartite graph $H$ has minimum degree at least $2$, then
$$\wT_H(2,0)\wT_H(0,2)\geqslant\wT_H(1,1)^2.$$
\end{Conj}

Clearly, Conjecture~\ref{degree-2} implies Conjecture~\ref{simple-cosimple}. At this point, we must admit that we lack examples.  Computational results suggest the following conjecture. (Here we only phrase the conjecture when the number of vertices is divisible by $4$, but there are similar conjectures when $n=4k+i$ for $i\in \{1,2,3\}$.)

\begin{Conj} 
Let $n=4k$, then the following bipartite graph minimizes $\wT_H(2,0)\wT_H(0,2)$ among bipartite graphs on $n$ vertices without isolated vertices: add a leaf vertex to each vertex of a $K_{k,k}$. 
\end{Conj}

\begin{figure}[h!]
\begin{tikzpicture}[scale=1.2]
\node[vertex] (a1) at (1,0) [circle,fill=black] {};
\node[vertex] (a2) at (2,0) [circle,fill=black] {};
\node[vertex] (a3) at (3,0) [circle,fill=black] {};
\node[vertex] (a4) at (4,0) [circle,fill=black] {};
\node[vertex] (a5) at (5,0) [circle,fill=black] {};
\node[vertex] (a6) at (6,0) [circle,fill=black] {};
\node[vertex] (b1) at (1,1) [circle,fill=black] {};
\node[vertex] (b2) at (2,1) [circle,fill=black] {};
\node[vertex] (b3) at (3,1) [circle,fill=black] {};
\node[vertex] (b4) at (4,1) [circle,fill=black] {};
\node[vertex] (b5) at (5,1) [circle,fill=black] {};
\node[vertex] (b6) at (6,1) [circle,fill=black] {};
\node[vertex] (c1) at (1,2) [circle,fill=black] {};
\node[vertex] (c2) at (2,2) [circle,fill=black] {};
\node[vertex] (c3) at (3,2) [circle,fill=black] {};
\node[vertex] (c4) at (4,2) [circle,fill=black] {};
\node[vertex] (c5) at (5,2) [circle,fill=black] {};
\node[vertex] (c6) at (6,2) [circle,fill=black] {};
\node[vertex] (d1) at (1,-1) [circle,fill=black] {};
\node[vertex] (d2) at (2,-1) [circle,fill=black] {};
\node[vertex] (d3) at (3,-1) [circle,fill=black] {};
\node[vertex] (d4) at (4,-1) [circle,fill=black] {};
\node[vertex] (d5) at (5,-1) [circle,fill=black] {};
\node[vertex] (d6) at (6,-1) [circle,fill=black] {};
\draw (a1) -- (b1);
\draw (a1) -- (b2);
\draw (a1) -- (b3);
\draw (a1) -- (b4);
\draw (a1) -- (b5);
\draw (a1) -- (b6);
\draw (a2) -- (b1);
\draw (a2) -- (b2);
\draw (a2) -- (b3);
\draw (a2) -- (b4);
\draw (a2) -- (b5);
\draw (a2) -- (b6);
\draw (a3) -- (b1);
\draw (a3) -- (b2);
\draw (a3) -- (b3);
\draw (a3) -- (b4);
\draw (a3) -- (b5);
\draw (a3) -- (b6);
\draw (a4) -- (b1);
\draw (a4) -- (b2);
\draw (a4) -- (b3);
\draw (a4) -- (b4);
\draw (a4) -- (b5);
\draw (a4) -- (b6);
\draw (a5) -- (b1);
\draw (a5) -- (b2);
\draw (a5) -- (b3);
\draw (a5) -- (b4);
\draw (a5) -- (b5);
\draw (a5) -- (b6);
\draw (a6) -- (b1);
\draw (a6) -- (b2);
\draw (a6) -- (b3);
\draw (a6) -- (b4);
\draw (a6) -- (b5);
\draw (a6) -- (b6);
\draw (c1) -- (b1);
\draw (c2) -- (b2);
\draw (c3) -- (b3);
\draw (c4) -- (b4);
\draw (c5) -- (b5);
\draw (c6) -- (b6);
\draw (a1) -- (d1);
\draw (a2) -- (d2);
\draw (a3) -- (d3);
\draw (a4) -- (d4);
\draw (a5) -- (d5);
\draw (a6) -- (d6);
\end{tikzpicture}
\caption{The conjectured minimizer graph for $\wT_H(2,0)\wT_H(0,2)$ on $24$ vertices.}
\end{figure}

If this conjecture is true, then the Merino-Welsh conjecture is true for matroids on $n$ elements when $n\leqslant 40$ and divisible by $4$, which shows why it is hard to find a counter example for the Merino-Welsh conjecture just by brute force computer search.

\bigskip

\noindent \textbf{Acknowledgment.} The author is very grateful to Csongor Beke, Ferenc Bencs, Gergely K\'al Cs\'aji and S\'ara Pituk for discussions on the topic of this paper. The author also thanks the very careful work of the two referees, they spotted many typos and gave very useful suggestions. 

\bibliography{references}

@article{welsh1999tutte,
  title={{The {T}utte polynomial}},
  author={Welsh, Dominic},
  journal={Random Structures \& Algorithms},
  volume={15},
  number={3-4},
  pages={210--228},
  year={1999},
  publisher={Wiley Online Library}
}

@article{brylawski1992tutte,
  title={{The {T}utte polynomial and its applications}},
  author={Brylawski, Thomas and Oxley, James},
  journal={Matroid applications},
  volume={40},
  pages={123--225},
  year={1992}
}

@article{crapo1969tutte,
  title={{The {T}utte polynomial}},
  author={Crapo, Henry H.},
  journal={Aequationes Mathematicae},
  volume={3},
  number={3},
  pages={211--229},
  year={1969},
  publisher={Springer}
}

@incollection{ellis2011graph,
  title={Graph polynomials and their applications I: The {T}utte polynomial},
  author={Ellis-Monaghan, Joanna A. and Merino, Criel},
  booktitle={Structural analysis of complex networks},
  pages={219--255},
  year={2011},
  publisher={Springer}
}

@article{thomassen2010spanning,
  title={Spanning trees and orientations of graphs},
  author={Thomassen, Carsten},
  journal={Journal of Combinatorics},
  volume={1},
  number={2},
  pages={101--111},
  year={2010},
  publisher={International Press of Boston}
}

@article{tutte1954contribution,
  title={A contribution to the theory of chromatic polynomials},
  author={Tutte, William Thomas},
  journal={Canadian Journal of mathematics},
  volume={6},
  pages={80--91},
  year={1954},
  publisher={Cambridge University Press}
}

@article{fortuin1971correlation,
  title={Correlation inequalities on some partially ordered sets},
  author={Fortuin, Cees M. and Kasteleyn, Pieter W. and Ginibre, Jean},
  journal={Communications in Mathematical Physics},
  volume={22},
  number={2},
  pages={89--103},
  year={1971},
  publisher={Springer}
}

@book{ellis2022handbook,
  title={{Handbook of the Tutte polynomial and related topics}},
  author={Ellis-Monaghan, Joanna A. and Moffatt, Iain},
  year={2022},
  publisher={CRC Press}
}

@article{ferroni2023merino,
  title={{The Merino--Welsh conjecture for split matroids}},
  author={Ferroni, Luis and Schr{\"o}ter, Benjamin},
  journal={Annals of Combinatorics},
  volume={27},
  number={3},
  pages={737--748},
  year={2023},
  publisher={Springer}
}

@article{chavez2011some,
  title={{Some inequalities for the Tutte polynomial}},
  author={Ch{\'a}vez-Lomel{\'\i}, Laura E. and Merino, Criel and Noble, Steven D. and Ram{\'\i}rez-Ib{\'a}{\~n}ez, Marcelino},
  journal={European Journal of Combinatorics},
  volume={32},
  number={3},
  pages={422--433},
  year={2011},
  publisher={Elsevier}
}

@article{merino2009note,
  title={{A note on some inequalities for the Tutte polynomial of a matroid}},
  author={Merino, Criel and Iba{\~n}ez, Marcelino and Rodr{\'\i}guez, M Guadalupe},
  journal={Electronic Notes in Discrete Mathematics},
  volume={34},
  pages={603--607},
  year={2009},
  publisher={Elsevier}
}

@article{jackson2010inequality,
  title={{An inequality for Tutte polynomials}},
  author={Jackson, Bill},
  journal={Combinatorica},
  volume={30},
  pages={69--81},
  year={2010},
  publisher={Springer}
}

@article{knauer2018tutte,
  title={{A Tutte polynomial inequality for lattice path matroids}},
  author={Knauer, Kolja and Mart{\'\i}nez-Sandoval, Leonardo and Alfons{\'\i}n, Jorge Luis Ram{\'\i}rez},
  journal={Advances in Applied Mathematics},
  volume={94},
  pages={23--38},
  year={2018},
  publisher={Elsevier}
}

@article{conde2009comparing,
  title={Comparing the number of acyclic and totally cyclic orientations with that of spanning trees of a graph},
  author={Conde, Rodolfo and Merino, Criel},
  journal={Int. J. Math. Com},
  volume={2},
  pages={79--89},
  year={2009}
}

@article{lin2013note,
  title={A note on spanning trees and totally cyclic orientations of 3-connected graphs},
  author={Lin, Fenggen},
  journal={Journal of Combinatorics},
  volume={4},
  number={1},
  pages={95--104},
  year={2013},
  publisher={International Press of Boston}
}

@article{noble2014merino,
  title={{The Merino--Welsh conjecture holds for series--parallel graphs}},
  author={Noble, Steven D. and Royle, Gordon F.},
  journal={European Journal of Combinatorics},
  volume={38},
  pages={24--35},
  year={2014},
  publisher={Elsevier}
}

@article{kung2021inconsequential,
  title={{Inconsequential results on the Merino-Welsh conjecture for Tutte polynomials}},
  author={Kung, Joseph P. S.},
  journal={arXiv preprint arXiv:2105.01825},
  year={2021}
}

@article{merino1999forests,
  title={Forests, colorings and acyclic orientations of the square lattice},
  author={Merino, Criel and Welsh, Dominic},
  journal={Annals of Combinatorics},
  volume={3},
  number={2-4},
  pages={417--429},
  year={1999},
  publisher={Springer}
}

@book{oxley1992matroid,
  title={Matroid theory},
  author={Oxley, James},
  year={1992},
  publisher={Oxford University Press}
}

@inproceedings{harris1960lower,
  title={A lower bound for the critical probability in a certain percolation process},
  author={Harris, Theodore E},
  booktitle={Mathematical Proceedings of the Cambridge Philosophical Society},
  volume={56},
  number={1},
  pages={13--20},
  year={1960},
  organization={Cambridge University Press}
}

@article{beke2024merino,
  title={{The Merino--Welsh conjecture is false for matroids}},
  author={Beke, Csongor and Cs{\'a}ji, Gergely K{\'a}l and Csikv{\'a}ri, P{\'e}ter and Pituk, S{\'a}ra},
  journal={Advances in Mathematics},
  volume={446},
  pages={109674},
  year={2024},
  publisher={Elsevier}
}

@article{beke2024permutation,
  title={{Permutation Tutte polynomial}},
  author={Beke, Csongor and Cs{\'a}ji, Gergely K{\'a}l and Csikv{\'a}ri, P{\'e}ter and Pituk, S{\'a}ra},
  journal={European Journal of Combinatorics},
  volume={120},
  pages={104003},
  year={2024},
  publisher={Elsevier}
}

@misc{JUP,
  author = {Csikv{\'a}ri, P{\'e}ter},
  title = {{Jupyter Notebook 'Merino--Welsh optimization'}},
  year = {2025},
  publisher = {GitHub},
  journal = {GitHub repository},
  howpublished = {\url{https://github.com/Merino-Welsh-conjecture/blob/main/Merino-Welsh_optimization.ipynb}}
}
\bibliographystyle{plain}

\newpage

\section{Appendix}

Below, one can find the tables belonging to the various ideas.

\subsection{Table for Idea 1.} Let $x:=\frac{3+\sqrt{5}}{2}\approx 2.6180$ and $s:=\frac{6+2\sqrt{5}}{7+3\sqrt{5}} \approx 0.7639$. Below we give the values of $G(d,x,s,\gamma_{x,s}(1))$.
\bigskip

\begin{table}[!h]
\begin{center}
\begin{tabular}{|c|c|} \hline
$d$ & $G(d,x,s,\gamma_{x,s}(1))$ \\ \hline
1 & 1.00000000000000 \\ \hline
2 & 1.15236921034711  \\ \hline
3 & 1.18525033351524  \\ \hline
4 & 1.18783805465536  \\ \hline
5 & 1.18125856327950  \\ \hline
6 & 1.17211090590244  \\ \hline
7 & 1.16272735027462  \\ \hline
8 & 1.15394531114347  \\ \hline
9 & 1.14602516001225  \\ \hline
10 & 1.13899595137302  \\ \hline
11 & 1.13279660374225  \\ \hline
$\infty$ & 1.05572809000084  \\ \hline
\end{tabular}
\caption{}
\end{center}
\end{table}
\bigskip

\noindent For $d=9$ we have
$\left(\frac{x-1}{x(d+2)}\right)^{1/(d-1)}\approx 0.6977$. This is larger than $\gamma_{x,s}(1)\approx 0.6909$. This means that the sequence $G(d,x,s,\gamma_{x,s}(1))$ is decreasing from $d=9$, so the smallest value of $G(d,x,s,\gamma_{x,s}(1))$ from that point on is $G(\infty,x,s)$. (As we can see it is actually a decreasing sequence already from $d=4$.)
\bigskip

\subsection{Table for Idea 2.}

Let $x=2.54$ and $s=0.76$. 

\begin{table}[!h]
\begin{center}
\begin{tabular}{|c|c|} \hline
$d$ & $G(d,x,s,\gamma_{x,s}(1))$ \\ \hline
1* & 1.00015021063798 \\ \hline
2 & 1.12628760116317 \\ \hline
3 & 1.16035420716839 \\ \hline
4 & 1.16413305093218 \\ \hline
5 & 1.15856178434317 \\ \hline
6 & 1.15024896467994 \\ \hline
7 & 1.14155987842924 \\ \hline
8 & 1.13336130037553 \\ \hline
9 & 1.12593636935915 \\ \hline
10 & 1.11933141925454 \\ \hline
11 & 1.11349857234605 \\ \hline
$\infty$ & 1.04089600000000 \\ \hline
\end{tabular}
\caption{The value at $d=1$ is not $G(1,x,s,\gamma_{x,s}(1))$ but $G(1,x,s,\gamma_{x,s}(2))$.}
\end{center}
\end{table}
\bigskip

\bigskip

\subsection{Table for Idea 3.}

Let $x=2.36$ and $s=0.78$. 

\begin{table}[!h]
\begin{center}
\begin{tabular}{|c|c|c|} \hline
$d$ & $G(d,x,s,\gamma_{x,s}(1))$ & $G(d,x,s,\gamma_{x,s}(1))G(1,x,s,\gamma_{x,s}(d))^{\min(2,d-1)}$ \\ \hline
2 & 1.06874465202436  & 1.00215345922882 \\ \hline
3 & 1.10815264651986 &  1.00086197145259 \\ \hline
4 & 1.11681913369317 & 1.02640833194772 \\ \hline
5 & 1.11511667337790 &1.03743759976094 \\ \hline
6 &1.10975849131510 &1.04182882517377 \\ \hline
7 &1.10330750172561 &1.04300865188036 \\ \hline
8 & 1.09680798068072 &1.04261001534334 \\ \hline
9 & 1.09068230467836 &1.04145583550358 \\ \hline
10 & 1.08508115163529 &1.03997408296620 \\ \hline
11 & 1.08003342261565 &1.03838980790695 \\ \hline
12 & 1.07551442455619 &1.03682008545765 \\ \hline
... &   ... & ... \\ \hline
43 & 1.03218107718904 & 1.01881210598816 \\ \hline
44 &1.03176319039565 & 1.01863210924050 \\ \hline
$\infty$ & 1.01337600000000 &1.00047892960579 \\ \hline
\end{tabular}
\caption{}
\end{center}
\end{table}

\newpage

\subsection{Table for Idea 4.}

Let $x=2.355$ and $s=0.78$.

\begin{table}[!h]
\begin{center}
\begin{tabular}{|c|c|c|} \hline
$d$ & $G_2(d,x,s,\gamma_{x,s}(1),\gamma_{x,s}(2))$ & $G_2(d,x,s,\gamma_{x,s}(1),\gamma_{x,s}(2))G(1,x,s,\gamma_{x,s}(d))^{\min(2,d-1)}$ \\ \hline
2 & 1.07641984643180 & 1.00750701821492  \\ \hline
3 & 1.11135556808369 & 1.00001551323253  \\ \hline
4 &1.12051595966294 &1.02591566665079   \\ \hline
5 &1.11815485337910 &1.03629473774471   \\ \hline
6 &1.11192862150501 &1.03985270835852    \\ \hline
7 &1.10468853332594 &1.04027529602599    \\ \hline
8 &1.09755066213067 &1.03926180701977    \\ \hline
9 &1.09093594198563 &1.03763439690775    \\ \hline
10 &1.08497152313990 &1.03579937829855   \\ \hline
11 &1.07965934020252 &1.03395657259571   \\ \hline
12 &1.07495066897719 &1.03220035567236   \\ \hline
... &  ... &  ... \\ \hline
98 &1.02084509090909 &1.00934521974652  \\ \hline
99 &1.02076182000000 &1.00930958644862  \\ \hline
$\infty$ &1.01251800000000 &1.00115825634209  \\ \hline
\end{tabular}
\caption{}
\end{center}
\end{table}

\end{document}